\documentclass[12pt]{amsart}
\usepackage{setspace}
\usepackage{geometry}
\usepackage{todonotes}
\usepackage{amsthm}
\usepackage{amsfonts, amssymb, amscd}
\numberwithin{equation}{section}
\usepackage[symbol]{footmisc}

\usepackage{bm}
\usepackage{bbm}
\usepackage{verbatim}
\usepackage{mathrsfs}
\usepackage{graphicx}
\usepackage{tikz-cd}
\usepackage{subcaption}
\usepackage{listings}
\usepackage{subfiles}
\usepackage[toc,page]{appendix}
\usepackage{mathtools}
\usepackage{comment}
\usepackage{enumerate}
\usepackage{enumitem}
\usepackage[all]{xy}
\usepackage{graphicx}
\graphicspath{{images/}}
\usepackage{appendix}
\usepackage[colorlinks=true, citecolor=green, linkcolor=blue, filecolor=magenta, urlcolor=red,backref=page]{hyperref}

\usepackage{pgfplots}
\pgfplotsset{compat=1.17}

\newcommand{\lf}{\lfloor}
\newcommand{\rf}{\rfloor}

\newcommand{\bQ}{\mathbb{Q}}
\newcommand{\bR}{\mathbb{R}}
\newcommand{\bZ}{\mathbb{Z}}
\newcommand{\bC}{\mathbb{C}}
\newcommand{\bN}{\mathbb{N}}
\newcommand{\bP}{\mathbb{P}}

\newcommand{\cC}{\mathcal{C}}

\newcommand{\cF}{\mathcal{F}}
\newcommand{\cG}{\mathcal{G}}

\newcommand{\cO}{\mathcal{O}}
\newcommand{\cP}{\mathcal{P}}

\newcommand{\qlin}{\sim_{\mathbb{Q}}}

\newcommand{\rvol}{\operatorname{vol}_{X|S}}

\newcommand{\Prime}{^{\prime}}
\newcommand{\NE}{\mathrm{NE}}

\newcommand{\Supp}{\mathrm{Supp}}

\newcommand{\Ivol}{\mathrm{Ivol}}

\newcommand{\vol}{\operatorname{vol}}

\newcommand{\Aut}{\operatorname{Aut}}
\newcommand{\Hilb}{\operatorname{Hilb}}
\newcommand{\Univ}{\operatorname{Univ}}

\newcommand{\lct}{\operatorname{lct}}

\newcommand{\Spec}{\operatorname{Spec}}

\newcommand{\mult}{\operatorname{mult}}
\theoremstyle{plain} 
\newtheorem{thm}{Theorem}[section] 
\newtheorem{lemma}[thm]{Lemma}
\newtheorem{prop}[thm]{Proposition}
\newtheorem{cor}[thm]{Corollary}
\theoremstyle{definition} 
\newtheorem{defn}[thm]{Definition} 

\newtheorem{nota}[thm]{Notation}

\newtheorem{rem}[thm]{Remark}

\begin{document}
	
	\title{Boundedness of klt Good Minimal Models}
	\author{Xiaowei Jiang}
	\email{jxw20@mails.tsinghua.edu.cn}
	\address{Yau Mathematical Sciences Center, Tsinghua University, Beijing, China}
	\date{\today}
	
	\thanks{}

	\begin{abstract}
For good minimal models with klt singularities, polarized by Weil divisors that are relatively nef and big over the bases of the Iitaka fibration, we show that, after fixing appropriate numerical invariants, they form a bounded family. As an application, we construct separated coarse moduli spaces for klt good minimal models polarized by line bundles.
		
	\end{abstract}
	
	\maketitle
	\tableofcontents
	\section{Introduction}
	Throughout this paper,  we  work over an algebraically closed field $\mathbbm{k}$ of characteristic zero.
	\vspace*{20pt}

	The central problem in birational geometry is the classification of algebraic varieties. According to the standard minimal model conjecture and the abundance conjecture, any variety $Y$ with mild singularities is birational to a variety $X$ such that either $X$ admits a Mori–Fano fibration $X \to Z$ or $X$ is a good minimal model, that is, $K_X$ is semiample. Therefore, canonically polarized varieties, Fano varieties, Calabi–Yau varieties, and their iterated fibrations play a central role in birational geometry.
    
One of the main problems in the classification of algebraic varieties is whether there are only finitely many families of such objects after fixing certain numerical invariants; in other words, whether they form a bounded family. Establishing the boundedness of a given class of varieties is a natural first step toward constructing the corresponding moduli space.
	
	For canonically polarized varieties, boundedness was established in \cite{haconBirationalAutomorphismsVarieties2013, haconACCLogCanonical2014, haconBoundednessModuliVarieties2018}, while for Fano varieties it was proved in \cite{birkarAntipluricanonicalSystemsFano2019, birkarSingularitiesLinearSystems2021}. However, for Calabi–Yau varieties there is no natural choice of polarization, and in general they are not bounded in the category of algebraic varieties. For example, projective K3 surfaces and abelian varieties of any fixed dimension are not bounded. Nevertheless, there has been recent progress toward the (birational) boundedness of fibered Calabi–Yau varieties and rationally connected Calabi–Yau varieties; see \cite{birkarSingularitiesFanoFibrations2023,birkarBoundednessEllipticCalabiYau2024, filipazziBoundednessEllipticCalabi2024,jiangBoundednessPolarizedLog2025,engelBoundednessFiberedKtrivial2025}.
	 
	When studying the moduli of Calabi–Yau varieties, one typically fixes a polarization despite its non-uniqueness. Recently, Birkar established the following boundedness results for polarized Calabi–Yau varieties, which are crucial for constructing moduli spaces of such polarized varieties.
	
\begin{thm}[\cite{birkarGeometryPolarisedVarieties2023}] \label{thm:bdd of pcy}
    Let $d \in \bN$, $u \in \bQ^{>0}$, and $\Phi \subset \bQ^{\geq 0}$ be a DCC set. Consider Calabi--Yau pairs $(X,B)$ and $\bQ$-Cartier Weil divisors $A$ on $X$. Then the following hold:

    \begin{enumerate}
        \item[(1)] (klt case) If
        \begin{itemize}
            \item $(X,B)$ is a klt pair of dimension $d$,
            \item the coefficients of $B$ are contained in $\Phi$,
            \item $A$ is a nef and big  divisor on $X$ such that $\vol(A) = u$,
        \end{itemize}
        then the set of such $(X,B)$ forms a bounded family.  If in addition  $A\geq 0$, then the set of such  $(X,B+A)$ also  forms a bounded family.

        \item[(2)] (slc case) If
        \begin{itemize}
            \item $(X,B)$ is an slc pair of dimension $d$,
            \item the coefficients of $B$ are contained in $\Phi$,
            \item $A$ is an ample divisor on $X$ such that $\vol(A) = u$,
            \item $A\geq 0$ does not contain any non-klt center of $(X, B)$,
        \end{itemize}
        then the set of such  $(X,B+A)$ forms a bounded family.
    \end{enumerate}
\end{thm}

    Since the boundedness results for good minimal models of maximal and minimal Kodaira dimension have been established, it remains to investigate good minimal models with intermediate Kodaira dimension. Recently, Birkar proved the following boundedness result for slc good minimal models polarized by effective Weil divisors that are relatively ample over the bases of the Iitaka fibration, and he constructed their projective coarse moduli spaces. 

\begin{thm}[\cite{birkarModuliAlgebraicVarieties2022}]\label{thm:bdd of smm}
Let $d \in \bN$, $\Phi \subset \bQ^{\geq 0}$ be a DCC set, $\Gamma \subset \bQ^{>0}$ be a finite set, and $\sigma \in \bQ[t]$ be a polynomial. Consider pairs $(X,B)$ and $\bQ$-Cartier Weil divisors $A$ on $X$ satisfying the following conditions:
\begin{itemize}
    \item $(X,B)$ is an slc pair of dimension $d$, 
   
    \item the coefficients of $B$ are contained in $\Phi$,
     \item $K_X+B$ is semiample, defining a contraction $f:X \to Z$, 
    \item $A$ is a divisor on $X$ that is ample over $Z$,
    \item $\vol(A|_F) \in \Gamma$, where $F$ is any general fiber of $f:X \to Z$ over any irreducible component of $Z$,
    \item $(K_X+B+tA)^d = \sigma(t)$, and
    \item $A\ge 0$ does not contain any non-klt center of $(X, B)$,
\end{itemize}
then the set of such  $(X,B+A)$ forms a bounded family.
\end{thm}

Theorem \ref{thm:bdd of smm} can be regarded as a relative version of Theorem \ref{thm:bdd of pcy}(2). One naturally wonders whether a relative version of Theorem \ref{thm:bdd of pcy}(1) exists; that is, for klt pairs $(X,B)$, the polarization $A$ need not be an effective divisor. This paper addresses this question and uses it to construct the moduli space of klt good minimal models of arbitrary Kodaira dimension, polarized by line bundles that are relatively ample over the bases of the Iitaka fibration (see Appendix \ref{sec:moduli}).

\begin{thm} \label{thm:boundedness}
Let $d \in \bN$, $\Phi \subset \bQ^{\geq 0}$ be a DCC set, $\Gamma \subset \bQ^{>0}$ be a finite set,  and $\sigma \in \bQ[t]$ be a polynomial. Let  $\cG_{klt}(d, \Phi, \Gamma, \sigma)$  be the set of  pairs $(X,B)$ and $\bQ$-Cartier Weil divisors $A$ on $X$ satisfying the following conditions:
\begin{itemize}
    \item $(X,B)$ is a klt pair of dimension $d$, 
   
    \item the coefficients of $B$ are contained in $\Phi$,
     \item $K_X+B$ is semiample, defining a contraction $f:X \to Z$, 
    \item $A$ is a divisor on $X$ that is nef and big over $Z$,
    \item $\vol(A|_F) \in \Gamma$, where $F$ is the general fiber of $f:X \to Z$, and
    \item $(K_X+B+tA)^d = \sigma(t)$.
\end{itemize}
Then the set of such  $(X,B)$ forms a bounded family.   If in addition  $A\geq 0$, then the set of such  $(X,B+A)$ also  forms a bounded family.
\end{thm}

	Recently, there have been some other related results on the (birational) boundedness   of klt good minimal models, see  \cite{filipazziInvariancePlurigeneraBoundedness2020, filipazziBoundednessEllipticCalabi2024,liBoundednessBaseVarieties2024,hashizumeBoundednessModuliSpaces2025,jiaoBoundednessPolarizedLog2025,zhuBoundednessStableMinimal2025,jiangBoundednessPolarizedLog2025}.

\begin{rem}
In Theorems \ref{thm:bdd of smm} and \ref{thm:boundedness}, conjecturally, the condition on $\vol(A|_F)$ can be removed; this is related to the effective b-semiampleness conjecture \cite[Conjecture 7.13]{prokhorovSecondMainTheorem2009}, and related discussions can be found in \cite{birkarBoundednessVolumeGeneralised2021, birkarModuliAlgebraicVarieties2022}. Note that the condition that $(X,B)$ is klt cannot be replaced by lc \cite[\S 4.2]{haconFailureBoundednessGeneralised2025}. While in Theorem \ref{thm:bdd of smm} the condition on fixing $\sigma(t)$ cannot be replaced by fixing only $\Ivol(K_X+B)$ \cite[\S 1]{birkarVariationsGeneralisedPairs2022}, one expects that in the klt case of Theorem \ref{thm:boundedness} such a replacement may be possible \cite{jiangBoundednessPolarizedLog2025}; however, in this case the polarization $A$ cannot be controlled as Theorem \ref{Thm C}.
\end{rem}

\noindent\textbf{Description of the proof.} 
\renewcommand{\thethm}{\Alph{thm}} 
\renewcommand{\theHthm}{\Alph{thm}} 
\setcounter{thm}{0} 

\begin{thm}[Boundedness of nef threshold]\label{Thm A}
Under the same assumptions as Theorem \ref{thm:boundedness}, 
there exists a positive rational number $\tau$, depending only on $(d, \Phi, \Gamma, \sigma)$, such that $K_X + B + tA$ is nef and big for all $0 < t < \tau$.

\end{thm}

\begin{thm}[Boundedness of pseudo-effective thereshold]\label{Thm B}
   Under the same assumptions as Theorem \ref{thm:boundedness}, there exists a positive rational number $ \lambda$, depending only on $(d, \Phi, \Gamma, \sigma)$, such that $K_X + B + tA$ is big for all $0 < t < \lambda$.
\end{thm}

\begin{thm}\label{Thm C}
Under the same assumptions as Theorem \ref{thm:boundedness}, there exist a natural number $r$ depending only on $(d, \Phi, \Gamma, \sigma)$ and a very ample divisor $H$ on $X$ such that
    $$
    H^d \leq r, \quad (K_X + B) \cdot H^{d-1} \leq r, \quad \text{and} \quad H - A \text{ is pseudo-effective}.$$
    
    In particular, by Lemma \ref{lem:bir19 2.20}, the set of such $(X,B)$ forms a bounded family.  If in addition  $A\geq 0$, then the set of such  $(X,B+A)$ also forms  a bounded family.

\end{thm}
\renewcommand{\thethm}{\thesection.\arabic{thm}}
\renewcommand{\theHthm}{\thesection.\arabic{thm}} 
\setcounter{thm}{2} 

It is clear that Theorem \ref{Thm A} implies Theorem \ref{Thm B}, while Theorem \ref{Thm C} yields Theorem \ref{thm:boundedness}.
The proof of Theorem \ref{Thm A}, \ref{Thm B} and \ref{Thm C} proceeds by induction on the dimension of $X$:

\begin{itemize}
  
	\item  $\text{Theorem \ref{Thm B}}_{d}$ +  $\text{Theorem \ref{Thm C}}_{d-1}$ $\implies$ $\text{Theorem \ref{Thm A}}_d$; cf.  (\ref{thm:bdd of tra smm with bdd peff threshold}).
	\item  $\text{Theorem \ref{Thm A}}_{d-1}$  $\implies$ $\text{Theorem \ref{Thm B}}_{d}$; cf. (\ref{prop: bdd peff threshold}).

	\item $\text{Theorem \ref{Thm A}}_{d}$ $\implies$  $\text{Theorem \ref{Thm C}}_{d}$; cf. (\ref{thm:bdd of tra strongly smm}).
\end{itemize}

\noindent	\textbf{Acknowledgement.}
The author expresses gratitude to Professor Junchao Shentu for discussions on \cite{shentuArakelovTypeInequalities2023}, and thanks his advisor Professor Caucher Birkar for generously sharing his survey note \cite{birkarBoundednessModuliAlgebraic2024}, which motivated the author to consider the problem in this paper. He also appreciates Professor Birkar for his constant support and helpful discussions. The author thanks Bingyi Chen, Minzhe Zhu, and Yu Zou for reading this paper and providing valuable suggestions, and Jia Jia, Junpeng Jiao, and Santai Qu for their useful comments.

	\section{Preliminaries}

	\textbf{Notations and conventions.}
	We collect some notations and conventions used in this paper.
	\begin{enumerate}
	
        
	\item A projective morphism $f:X\to Z$ of schemes is called a \textit{contraction} if $f_*\cO_X=\cO_Z$ ($f$ is not necessarily birational). In particular, $f$ is surjective with connected ﬁbers.

		\item Suppose that $X$ is a normal variety.  Let $D_1$ and  $D_2$ be $\bQ$-Cartier $\bQ$-divisors on $X$. We say that $D_1$ and $D_2$ are $\bQ$-linear equivalent, denoted by $D_1\qlin D_2$, if there exists $m\in\bZ_{>0}$ such that $mD_1$ and $mD_2$ are Cartier and $mD_1\sim mD_2$. Moreover,
		fixed $l\in\bZ_{>0}$, the notation $D_1\sim_l D_2$ means that $lD_1\sim lD_2$.

        \item  Let $f : X \to Z$ be a morphism between  normal varieties, and let $M$ and $L$ be $\bQ$-Cartier divisors on $X$. We say $M\sim L/Z$ (resp. $M \sim_{\bQ}  L/Z$) if there is a Cartier (resp. $\bQ$-Cartier)  divisor $N$ on $Z$ such that $M - L \sim f^*N$ (resp.  $M - L \sim_{\bQ} f^*N$). 
        
        	\item We say that a set  $\Phi \subset\bQ$ satisﬁes the \textit{descending chain condition}  (DCC, for short) if $\Phi$  does not contain any strictly decreasing inﬁnite sequence. Similarly, we say that a set  $\Phi \subset\bQ$  satisﬁes the \textit{ascending chain condition}  (ACC, for short) if $\Phi$ does not contain any strictly increasing inﬁnite sequence.
		
		\item     Let $X$ be a normal variety. A \textit{b-divisor}  over $X$ is a collection of $\bQ$-divisors $M_Y$ on $Y$ for each birational contraction $Y \to X$ from a normal variety that are compatible with respect to pushdown, that is, if $Y' \to X$ is another birational contraction and $\psi: Y' \dashrightarrow  Y$ is a morphism, then $ \psi_*M_{Y'} = M_Y$.

    	A b-divisor is \textit{b-$\bQ$-Cartier} if there is a birational contraction $Y\to X$ such that 
 $M_{Y}$ is $\bQ$-Cartier and $M_{Y'}$ is the pullback of $M_{Y}$ on $Y'$ for any birational contraction $Y'\to Y$.
In this case, we say that the b-divisor descends to $Y$, and it is represented by $M_Y$.

	\end{enumerate}
	
\subsection{(Generalised) pairs and singularities}
\begin{defn}[Generalised pairs]
    A \emph{generalised pair} $(X, B, M)$ consists of:
    \begin{itemize}
        \item a normal projective variety $X$,
        \item an effective $\bQ$-divisor $B \geq 0$ on $X$, and
        \item a b-$\bQ$-Cartier b-divisor $M$ over $X$, represented by a projective birational morphism $f \colon X' \to X$ and a $\bQ$-Cartier $\bQ$-divisor $M'$ on $X'$ such that:
        \begin{enumerate}
            \item $M'$ is nef, and
            \item $K_X + B + M$ is $\bQ$-Cartier, where $M := f_* M'$.
        \end{enumerate}
    \end{itemize}

We say that $(X, B + M)$ is a generalised pair with nef part $M'$.  
 
\end{defn}
Let $D$ be a prime divisor over $X$. Replace $X'$ with a log resolution of $(X,B)$ such that $D$ is a prime divisor on $X'$. We can write  
\[
K_{X'}+B'+M'=\pi^*(K_X+B+M).
\]  
Then the \textit{generalised log discrepancy} of $D$ is defined as  
\[
a(D,X,B,M)=1-\mult_D B'.
\]

We say that $(X,B+M)$ is \textit{generalised klt} (resp. \textit{generalised lc}, \textit{generalised $\epsilon$-lc}) if $a(D,X,B,M)>0$ (resp. $a(D,X,B,M)\geq0$, $a(D,X,B,M)\geq \epsilon$) for every prime divisor $D$ over $X$.

If $M=0$, then we say $(X,B)$ is a \textit{pair}, and we define its singularities similarly.

\begin{defn}[Lc threshold of $\bQ$-linear systems]
    Let $(X, B)$ be an lc pair.  
    The \textit{lc threshold} of a $\bQ$-Cartier $\bQ$-divisor $L \geq 0$ with respect to $(X, B)$ is defined as
    $$
    \mathrm{lct}(X,B,L) := \sup \{\, t \in \bR \mid (X, B + tL) \text{ is lc} \,\}.
    $$
    Now let $H$ be a $\bQ$-Cartier $\bQ$-divisor.  
    The $\bQ$-linear system of $H$ is
    $$
    |H|_{\bQ} := \{\, L \geq 0 \mid L \sim_{\bQ} H \,\}.
    $$
    We then define the \textit{lc threshold} of $|H|_{\bQ}$ with respect to $(X, B)$ (also called the \emph{global lc threshold} or \emph{$\alpha$-invariant}) as
    $$
    \mathrm{lct}(X, B, |H|_{\bQ}) := \inf \{\, \mathrm{lct}(X, B, L) \mid L \in |H|_{\bQ} \,\},
    $$
    which is equivalent to
    $$
    \sup \{\, t \in \bR \mid (X, B + tL) \text{ is lc for every } L \in |H|_{\bQ} \,\}.
    $$
\end{defn}
\begin{defn}[Good minimal models]
Let $\phi : X \dashrightarrow X^m$ be a projective birational contraction between normal projective varieties.  
Suppose that $(X, B)$ and $(X^m, B^m)$ are lc pairs, where $B^m = \phi_* B$.  
If 
\[
a(E, X, B) > a(E, X^m, B^m)
\] 
for all prime $\phi$-exceptional divisors $E \subset X$, $X^m$ is $\bQ$-factorial, and $K_{X^m} + B^m$ is nef,  
then we say that $\phi : X \dashrightarrow X^m$ is a \emph{minimal model} of $(X, B)$.  
If, further, $K_{X^m} + B^m$ is semiample, then the minimal model $\phi : X \dashrightarrow X^m$ is called a \emph{good minimal model}.
\end{defn}
\subsection{Canonical bundle formula}\label{sec:cbf}

    An \emph{lc-trivial fibration} $f \colon (X, B) \to Z$ consists of a projective surjective morphism $f \colon X \to Z$ with connected fibers between normal varieties such that
    \begin{itemize}
        \item $(X, B)$ is an lc pair, and
        \item there exists a $\bQ$-Cartier $\bQ$-divisor $L_Z$ on $Z$ such that
        \[
            K_X + B \sim_{\bQ} f^* L_Z .
        \]
    \end{itemize}

Let $f : (X, B) \to Z$ be an lc-trivial fibration with $\dim Z > 0$.  
Fix a prime divisor $D$ on $Z$, and let $t_D$ be the lc threshold of $f^*D$ with respect to $(X,B)$ over the generic point of $D$.  
Define
\[
B_Z := \sum (1-t_D) D,
\]
where the sum runs over all prime divisors on $Z$.  
Set
\[
M_Z := L_Z - (K_Z + B_Z),
\]
so that
\[
K_X + B \sim_{\bQ} f^*(K_Z + B_Z + M_Z).
\]
We call $B_Z$ the \emph{discriminant divisor} and $M_Z$ the \emph{moduli divisor} of adjunction.  
Note that $B_Z$ is uniquely determined, whereas $M_Z$ is determined only up to $\bQ$-linear equivalence.

Now let $\phi : X' \to X$ and $\psi : Z' \to Z$ be birational morphisms from normal projective varieties, and assume that the induced map $f' : X' \dashrightarrow Z'$ is a morphism.  
Let $K_{X'} + B'$ be the pullback of $K_X + B$ to $X'$.  
Similarly, we can define a discriminant divisor $B_{Z'}$ on $Z'$ and, setting $L_{Z'} := \psi^* L_Z$, obtain a moduli divisor $M_{Z'}$ such that
\[
K_{X'} + B' \sim_{\bQ} f'^*(K_{Z'} + B_{Z'} + M_{Z'}),
\]
with $B_Z = \psi_* B_{Z'}$ and $M_Z = \psi_* M_{Z'}$.  

In particular, the lc-trivial fibration $f : (X, B) \to Z$ induces b-$\bQ$-divisors $B$ and $M$ on $Z$, called the \emph{discriminant} and \emph{moduli} b-divisors, respectively.
\begin{thm}[\cite{ambroShokurovBoundaryProperty2004,fujinoModuliBdivisorsLctrivial2014}]
    With the above notation and assumptions.
    If $Z '\to  Z$ is a high resolution, then $M_{Z'}$ is nef and for any birational morphism $Z'' \to Z'$ from a normal projective variety, $M _{Z''}$ is the pullback of $M_{Z'}$. In particular, we can view $(Z,B_Z+M_Z)$ as a generalized pair with nef part $M_{Z'}$.
\end{thm}

	\subsection{Volume of divisors} We recall the definition of various types of volumes for divisors. In this paper, we mainly consider $\bQ$-divisors. However, for the proof of Proposition \ref{prop: bdd peff threshold}, we need to deal  with $\bR$-divisors.
	\begin{defn}[Volumes]
		Let $X$ be a normal  irreducible projective variety of dimension $d$, and let $D$ be an $\bR$-divisor on $X$. The \textit{volume} of $D$ is
		$$\vol(X,D)=\limsup_{m\to \infty}\frac{d!h^0(X,\cO_X(\lf mD\rf))}{m^d}.$$
		
	\end{defn}
\begin{defn}[Iitaka volumes]\label{def:Ivol}
Let $X$ be a normal irreducible projective variety of dimension $d$, and let $D$ be an $\mathbb{R}$-divisor on $X$.  
The \emph{Iitaka volume} of $D$, denoted by $\Ivol(D)$, is defined as
\[
\Ivol(D) :=
\begin{cases}
\displaystyle \limsup_{m\to\infty} \frac{\kappa(D)!\, h^0\!\left(X, \mathcal{O}_X\!\left(\lfloor mD \rfloor\right)\right)}{m^{\kappa(D)}} & \text{if } \kappa(D) \ge 0, \\[1em]
0 & \text{if } \kappa(D) = -\infty,
\end{cases}
\]
where $\kappa(D)$ denotes the Iitaka dimension of $D$.  
By convention, when $\kappa(D) = 0$ we set $\kappa(D)! = 1$, so in this case $\Ivol(D) = 1$.
\end{defn}

    If $f \colon X \to Z$ is a contraction and $D \sim_{\mathbb{Q}} f^{*} D_Z$ for some big $\mathbb{Q}$-divisor $D_Z$ on $Z$, then
$\Ivol(D) = \vol(D_Z)$.
	\begin{defn}[Restricted volumes]\label{def:restrict vlume}
		Let $X$ be a normal  irreducible projective variety of dimension $d$, and let $D$ be an $\bQ$-divisor on $X$. 
		Let $S\subset X$ be a normal irreducible subvariety of dimension $n$. 
		Suppose that $S$ is not contained in the augmented base locus $ \textbf{B}_+(D)$. 
		Then the \textit{restricted volume of $D$ along $S$} is
		$$\rvol(D)=\limsup_{m\to \infty}\frac{n!(\dim \mathrm{Im} (H^0(X,\cO_X(\lf mD\rf))\to H^0(S,\cO_S(\lf mD|_S\rf))))}{m^n}.$$
	\end{defn}
	
	For the precise definition of the \textit{augmented base locus} $ \textbf{B}_+(D)$, see \cite{einAsymptoticInvariantsBase2006}. In this paper, we only use the fact that $ \textbf{B}_+(D)$ is a Zariski-closed subset of $X$ such that $ \textbf{B}_+(D)\subsetneq X$ if and only if $D$ is big.
	The restricted volume $\rvol(D)$ measures asymptotically the number of sections of the restriction $\cO_S(\lf mD|_S\rf)$ that can be lifted to $X$. If $D$ is ample, then the restriction maps are eventually surjective, and hence
	$$\rvol(D)=\vol(D|_S).$$
	In general,  it can happen that $\rvol(D)<\vol(D|_S)$.
	\begin{thm}[{\cite[Corollary 4.27]{lazarsfeldConvexBodiesAssociated2009}}]\label{thm:LM09 cor 4.27}
		Let $X$ be an irreducible projective variety of dimension $d$,  and let $S \subset X$ be an irreducible (and reduced) Cartier divisor on $X$. Suppose that $D$ is a big $\bR$-divisor such that 	 $S\nsubseteq \textup{\textbf{B}}_+(D)$. Then the function $t\mapsto \vol(D+tS)$
		is diﬀerentiable at $t = 0$, and
		$$\frac{\mathrm{d}  }{\mathrm{d} t}(\vol(D+tS))\bigg|_{t=0}=d\rvol(D).$$
	\end{thm}
	By \cite[Remark 4.29]{lazarsfeldConvexBodiesAssociated2009}, volume function has continuous partials in all directions at any point $D\in \mathrm{Big}(X)$, i.e., the function
	$\vol:\mathrm{Big}(X)\to \bR$
	is $\cC^1$.

\subsection{Bounded family of pairs}
\begin{defn}[Bounded families of couples and pairs]
A \emph{couple} consists of a projective normal variety $X$ and a reduced divisor $D$ on $X$.  
We call $(X, D)$ a couple rather than a pair because $K_X + D$ is not assumed to be $\bQ$-Cartier and $(X, D)$ is not assumed to have good singularities.  

Two couples $(X, D)$ and $(X', D')$ are said to be \emph{isomorphic} if there exists an isomorphism $X \to X'$ that maps $D$ onto $D'$.  

Let $\cP$ be a set of couples. We say that $\cP$ is \emph{bounded} if the following conditions hold:  
\begin{itemize}
    \item There exist finitely many projective morphisms $V^i \to T^i$ of varieties,  
    \item $C^i$ is a reduced divisor on $V^i$, and  
    \item for each $(X, D) \in \cP$, there exist an index $i$, a closed point $t \in T^i$, and an isomorphism $\phi: V^i_t \to X$ such that $(V^i_t, C^i_t)$ is a couple and $\phi_* C^i_t \ge D$.  
\end{itemize}

A set of projective pairs $(X, B)$ is said to be \emph{bounded} if the set of couples $(X, \Supp B)$ forms a bounded family.
\end{defn}

Boundedness for couples is equivalent to the following criterion.
\begin{lemma}[{\cite[Lemma 2.20]{birkarAntipluricanonicalSystemsFano2019}}] \label{lem:bir19 2.20}
    Let $d,r\in\bN$. Assume $\cP $ is a set of couples $(X, D)$ where $X$ is of dimension $d$ and there is a very ample divisor $H$ on $X$ with $H^d\leq r$ and $H^{d-1}\cdot D\leq r$. Then $\cP$ is bounded.
\end{lemma}

\begin{lemma}[{\cite[Lemma 4.6]{birkarModuliAlgebraicVarieties2022}}]\label{lem:bir22 4.6}
Let $d,r\in\bN$ and let $\Phi\subset \bQ^{\geq 0}$ be a DCC set. Then there 
exists $l\in \bN$ satisfying the following. Assume 
\begin{itemize}
\item $X$ is a normal projective variety of dimension $d$, 
\item $H$ is a very ample divisor, 
\item $B$ is a divisor with coefficients in $\Phi$, and
\item $H^d\le r$ and $B\cdot H^{d-1}\le r$.  
\end{itemize}
Then $lH-B$ is pseudo-effective. 
\end{lemma}

The following theorem is one of the main ingredients in the proof of Theorem \ref{Thm A}. 
We emphasise that it imposes no restriction on the coefficients of $B$ and $M$.
\begin{thm}[{\cite[Theorem 1.8]{birkarSingularitiesLinearSystems2021}}]
\label{thm:bir21 1.8}
        Let $d,r\in\bN$ and $\epsilon\in\bQ^{>0}$.
		Then there is a positive rational number $t$  depending only on $d,r,\epsilon$,
		satisfying the following. Assume
		\begin{itemize}
			\item $(X,B)$ is projective $\epsilon$-lc of dimension $d$,
			\item $H$ is a very ample divisor on $X$ with $H^ d \leq  r$,
			\item $H -B$ is pseudo-eﬀective, and
			\item $M \geq 0$ is an $\bQ$-Cartier $\bQ$-divisor with $H -M$ pseudo-eﬀective.
		\end{itemize}
		Then
		$$\lct(X, B, | M | _{\bQ} ) \geq \lct(X, B, | H |_{\bQ} ) \geq t.$$
	\end{thm}

We will use the following boundedness result for polarized nef pairs to deduce Theorem \ref{Thm C} from Theorem \ref{Thm A}.
\begin{thm} [{\cite[Theorem 1.5]{birkarGeometryPolarisedVarieties2023}}]
\label{thm:bir23 1.5}
		Let $d\in\bN$, $\delta,v\in\bQ^{>0}$. Consider pairs $(X,B)$ and nef and big Weil divisors $N$ on $X$ such that
		\begin{itemize}
			\item $(X,B)$ is projective $\epsilon$-lc of dimension $d$,
			\item the coefﬁcients of $B$ are in $\{ 0 \} \cup [ \delta,\infty)$,
			\item  $K_ X + B$ is nef,

			\item $\vol(K_ X + B + N) \leq v$.
		\end{itemize}
		Then the set of such $(X,B)$ forms a bounded family. If in addition $N \geq 0$, then the set of such $(X,B + N)$ forms a bounded family.
	\end{thm}



    \section{Boundedness}
	In this section, we  prove Theorem \ref{thm:boundedness}.
    
	\subsection{Boundedness of generalised pairs on bases of ﬁbrations}
In this subsection, we consider the set of good minimal models whose general fibers of the Iitaka fibration belong to a bounded family and whose Iitaka volume is fixed.

     \begin{defn}
       Let $d \in \bN$, $\Phi \subset \bQ^{\geq 0}$ be a DCC set, and $u,v\in \bQ^{>0}$. Let $\cG_{klt}(d,\Phi, u,v)$ be the set of   $(X,B)$ and $\bQ$-Cartier Weil divisors $A$ on $X$ satisfying the following conditions:
\begin{itemize}
    \item $(X,B)$ is a klt pair of dimension $d$, 
   
    \item the coefficients of $B$ are contained in $\Phi$,
     \item $K_X+B$ is semiample, defining a contraction $f:X \to Z$, 
    \item $A$ is a divisor on $X$ that is nef and big over $Z$,
    \item $\vol(A|_F)= u$, where $F$ is the general fiber of $f:X \to Z$, and
    \item $\Ivol(K_X+B) = v$.
\end{itemize}  
     \end{defn}

Since $K_X+B$ is semiample, there exists a contraction $f:X \to Z$ onto a normal variety $Z$.
By the canonical bundle formula in \S \ref{sec:cbf}, we can write
$$K_X+B \qlin f^*(K_Z+B_Z+M_Z),$$
and we may then regard $(Z, B_Z+M_Z)$ as a generalised pair with ample $K_Z+B_Z+M_Z$, that is, a \emph{generalised log canonical (lc) model}.

\begin{lemma}\label{lem: finiteness of discrepancy}
Let $d \in \bN$, $\Phi \subset \bQ^{\geq 0}$ be a DCC set, and $u, v \in \bQ^{>0}$. 
Then there exist $p, q \in \bN$ depending only on $(d, \Phi, u)$,  
and $l \in \bN$, $\epsilon \in \bQ^{>0}$ depending only on $(d, \Phi, u, v)$,   such that for any
$$(X, B), A \to Z \in \cG_{klt}(d, \Phi, u, v),$$
the following hold:
\begin{enumerate}
    \item  
    We have an adjunction formula
    \[
    K_X + B \sim_q f^*(K_Z + B_Z + M_Z),
    \]
    where $p M_{Z'}$ is Cartier on some high resolution $Z' \to Z$.

    \item  
    The pair $(X, B)$ is $\epsilon$-lc, and $lB$ is a Weil divisor.
\end{enumerate}
\end{lemma}

\begin{proof}
Replacing $X$ with the ample model of $A$ over $Z$, we may assume that $A$ is ample over $Z$.  
Applying \cite[Corollary~1.4]{birkarGeometryPolarisedVarieties2023} to $(F, B|_F)$ and $A|_F$, there exists $m \in \bN$, depending only on $d$ and $\Phi$, such that $H^0(F, \cO_X(mA|_F)) \neq 0$.  
Hence $mA \sim G$ for some Weil divisor $G$. Replacing $A$ with the horizontal part of $G$, we may assume that $A$ is effective.  

Applying \cite[Lemma~7.4]{birkarBoundednessVolumeGeneralised2021} yields integers $p, q$ satisfying (1).  
Moreover, by \cite[Lemma~8.2]{birkarBoundednessVolumeGeneralised2021},   the set of log discrepancies
\[
\{ a(D, X, B) \leq 1 \mid D \text{ a prime divisor over } X \}
\]
is finite, and hence (2) holds. Note that the proof of \cite[Lemma~8.2]{birkarBoundednessVolumeGeneralised2021} uses $A$ only in the relative sense over $Z$. 
\end{proof}

\begin{defn}[{\cite[Definition 1.1]{birkarBoundednessVolumeGeneralised2021}}]
		Let $d\in \bN$, $\Phi\subset \bQ^{\geq 0} $ be a DCC set, and $v\in \bQ ^{>0}$. 
		Let $\cF_{gklt}(d,\Phi,v)$ be the set of projective generalised pairs $(X, B +M)$ with nef part $M\Prime$ such that
		\begin{itemize}
			\item $(X, B + M)$ is generalised klt of dimension $d$,
			\item the coeﬃcients of $B$ are in $\Phi$,
			\item $M\Prime = \sum \mu_i M_i\Prime$ where $\mu_i\in\Phi$ and $M_i\Prime$ are nef Cartier, and
			\item $K_X+ B + M$ is ample with volume $\vol(K_X+ B + M) = v$.
		\end{itemize}
	\end{defn}
Now we can prove the boundedness of bases of Iitaka fibrations with their induced generalised pair structure under natural assumptions.
	\begin{thm}[ \cite{birkarBoundednessVolumeGeneralised2021}]\label{thm:bdd of base}
		Let $d\in \bN$, $\Phi\subset \bQ^{\geq 0} $ be a DCC set, and $u,v\in \bQ ^{>0}$. 
		Then there exists $l \in \bN$ depending only on $d,\Phi,u,v$ such that for any
		$$(X, B), A \to Z\in \cG_{klt}(d,\Phi,u,v),$$
		we can write an adjunction formula
		$$K_X+B\sim_l f^*(K_Z + B_Z + M_Z )$$
		such that the corresponding set of generalized pairs $(Z, B_Z + M_Z )$ forms a bounded family. Moreover, $l(K_Z + B_Z + M_Z )$ is very ample.
	\end{thm}
	\begin{proof}
By Lemma \ref{lem: finiteness of discrepancy} (1),
		there exist $p,q \in \bN$ depending only on $d,\Phi,u$  such that 
		we can write an adjunction formula
		$$K_X+B\sim_q f^*(K_Z + B_Z + M_Z ),$$
		where $pM_{Z\Prime}$ is Cartier  on some higher resolution $Z\Prime\to Z$.
        
		By definition of the discriminant part of the canonical bundle formula and the ACC for lc thresholds \cite[Theorem 1.1]{haconACCLogCanonical2014}, we see that the coeﬃcients of $B_Z$ belong to a DCC subset of $\bQ^{>0}$ depending only on $d$ and $\Phi$, which we denote by $\Psi$.  Moreover, $(Z, B_Z+ M_Z)$ is generalised klt pair and
		$$\Ivol(K_X + B) = \vol(K_Z + B_Z + M_Z ) = v.$$
		Adding $\frac{1}{p}$, we can assume $\frac{1}{p}\in\Psi $, we see that 
		$$(Z,B_Z+M_Z)\in \cF_{gklt}(\dim Z,\Psi,v).$$
In the proof of \cite[Theorem~1.4]{birkarBoundednessVolumeGeneralised2021}, a divisor $\Theta$ is constructed such that 
\[
l(K_X+\Theta) \sim l(1+t)(K_X+B+M)
\]
is ample, $(X, \Theta)$ is $\epsilon$-lc, and the coefficients of $\Theta$ belong to a fixed DCC set $\Psi'$. Here $l \in \bN$, $t, \epsilon \in \bQ^{>0}$, and $\Psi' \subset \bQ^{>0}$ depend only on $(d, \Phi, u, v)$. Moreover, $(X, \Theta)$ is log birationally bounded.
By \cite[Theorem~1.6]{haconACCLogCanonical2014}, $(X, \Theta)$ belongs to a bounded family.  
Thus, we may replace $l$ so that both $l(K_X+\Theta)$ and $l(K_X+B+M)$ are very ample.  
Hence, the set of generalised pairs $(Z, B_Z + M_Z)$ forms a bounded family.
         Replacing $q,l$ with $ql$, we conclude the proof.
	\end{proof}

			

	\subsection{Boundedness  of nef threshold}

	In this subsection, we show that the nef threshold of $K_X+B$ with respect to $A$ is bounded for all 
$$(X,B), A \to Z \in \cG_{klt}(d, \Phi, \Gamma, \sigma).$$
We follow the argument of \cite[Theorem 4.1]{birkarModuliAlgebraicVarieties2022} with some modifications. The main difference is that, since $A$ may not be an effective divisor in our situation, we cannot directly apply the cone theorem to bound the nef threshold.

Therefore, we first assume that $K_X + B + \lambda A$ is big for some natural mumber $\alpha$ and rational number $\lambda \in [0,1]$. We can then replace $K_X + B + \lambda A$ by an effective $\bQ$-divisor $E$, but this loses control of the coefficients of $E$. For this reason, we require a stronger boundedness result on singularities in Theorem~\ref{thm:bir21 1.8} compared to \cite[Lemma 4.7]{birkarModuliAlgebraicVarieties2022}. To make the induction argument go through, we also need to show that $H - A$ is pseudo-effective, as in Theorem~\ref{Thm C}.
	
	\begin{prop}
	\label{thm:bdd of tra smm with bdd peff threshold}
		 Theorem $\ref{Thm B}_{d}$ and  Theorem $\ref{Thm C}_{d-1}$ imply Theorem $\ref{Thm A}_d$.
	\end{prop}
    
	\begin{proof}
		We proceed by  induction on the dimension of $X$.
		
	\textit{Step 1.} For each  
$$(X, B), A \to Z \in \cG_{klt}(d,\Phi,\Gamma,\sigma),$$  
we have  
$$\sigma(t) = (K_X+B+tA)^d = \sum_{i=0}^{d} \binom{d}{i} (K_X + B)^{d-i} \cdot A^i \, t^i,$$  
so the intersection numbers \((K_X + B)^{d-i} \cdot A^i\) are determined by \(d\) and \(\sigma\) for each \(0 \le i \le d\). In particular, for a general fiber \(F\) of \(X \to Z\),  
$$\Ivol(K_X + B) \cdot \vol(A|_F) = (K_X+B)^{\dim Z} \cdot A^{d-\dim Z}$$  
is a fixed number depending only on \(d\) and \(\sigma\). Since \(\vol(A|_F)\) belongs to the finite set \(\Gamma\), there are only finitely many possibilities for \(\Ivol(K_X + B)\). Therefore, we may fix both  
$$u := \vol(A|_F) \quad \text{and} \quad v := \Ivol(K_X + B).$$

		\textit{Step 2.} By Theorem \ref{Thm B} and Lemma \ref{lem: finiteness of discrepancy} (2),  we may choose  $\alpha\in\bN$ depending only on $d,\Phi,u,v,\lambda$ such that  $\alpha (K_X+B+\frac{\lambda}{2} A)$ is a big Weil divisor.  Moreover, 
			\begin{equation}\nonumber
\begin{aligned}
&\vol\Big(K_X + B + t \alpha (K_X + B + \tfrac{\lambda}{2} A)\Big)\\
=& \vol\Big((1 + t \alpha)(K_X + B) + \tfrac{t \alpha \lambda}{2} A\Big) \\
=& \Big((1 + t \alpha)(K_X + B) + \tfrac{t \alpha \lambda}{2} A\Big)^d
\end{aligned}
\end{equation}
		is a polynomial $\gamma$ in $t$ whose coeﬃcients are uniquely determined by the intersection numbers $(K_X + B)^{ d-i }\cdot A^ i $, $\alpha$ and  $\lambda$. Therefore, $\gamma$ is determined by $d, \Phi,\Gamma, \sigma, \lambda$.
		
		Replacing $A,u,\sigma$ with $\alpha (K_X+B+\frac{\lambda}{2} A), (\frac{\alpha\lambda}{2} )^{\dim F}u, \gamma$, we may assume that $A$ is a big Weil divisor.

		\textit{Step 3.} 
Since when $\dim X = 1$, $K_X + B + tA$ is always ample, and when $\dim Z = 0$, $K_X + B + tA$ is nef and big for all $0 < t < 1$, we may assume that $\dim X \ge 2$ and $\dim Z\geq 1$.

We claim that it suffices to find $\tau \in (0,1]$, depending only on $d, \Phi, \Gamma, \sigma$, such that $K_X + B + \tau A$ is nef.  
Indeed, once such a $\tau$ is found, $K_X + B + tA$ is nef and big for any $t \in (0, \tau)$.  
Since $A$ is nef and big over $Z$, by the base point free theorem it is semiample over $Z$, so we may pick $0 < t' \ll t$ such that $K_X + B + t'A$ is nef and big.  
Then $K_X + B + tA$ is a positive linear combination of $K_X + B + t'A$ and $K_X + B + \tau A$, and hence is nef and big.  

We aim to find such a $\tau$ in the subsequent steps.
		
			\textit{Step 4.}
		 By Theorem \ref{thm:bdd of base}, there exists $l \in \bN$ depending only on $d, \Phi, u, v$ such that we can write an adjunction formula
		$$K_X+B\sim_l f^*(K_Z + B_Z + M_Z )$$
		and  the generalised klt pair $(Z, B_Z + M_Z )$ belongs to a bounded family. Moreover,
		$$L := l(K_Z + B_Z + M_Z )$$
		is very ample.
		
Let $T$ be a general member of $|L|$, and let $S$ be its pullback to $X$.  
Define
\[
K_S + B_S := (K_X + B + S)|_S
\]
and set $A_S := A|_S$.  
Then
\[
(S, B_S),\ A_S \to T \in \cG_{klt}\bigl(d-1, \Phi, \Gamma, \psi\bigr)
\]
for some polynomial $\psi(t)$ depending only on $(d, \Phi, \Gamma, \sigma)$.

Indeed, we may choose a general $T \in |L|$ such that $A|_S$ is nef and big over $T$ and $(X, B+S)$ is plt.  
Hence $(S, B_S)$ is a projective klt pair, and $K_S + B_S$ is semi-ample, defining the contraction $g \colon S \to T$.  
If $G$ is a general fibre of $S \to T$, then
\[
\vol(A_S|_G) = \vol(A|_G) = u,
\]
since $G$ is among the general fibres of $X \to Z$. Moreover,
\begin{align*}
\psi(t) 
  &= (K_S + B_S + tA_S)^{d-1} \\
  &= \bigl((K_X + B + S + tA)|_S\bigr)^{d-1} \\
  &= (K_X + B + S + tA)^{d-1} \cdot S \\
  &= \bigl((l+1)(K_X + B) + tA\bigr)^{d-1} \cdot S \\
  &= \bigl((l+1)(K_X + B) + tA\bigr)^{d-1} \cdot l(K_X + B),
\end{align*}
which is a polynomial in $t$ whose coefficients are uniquely determined by the intersection numbers 
$(K_X + B)^{d-i} \cdot A^i$
and by $l$, and hence depend only on $d$, $\sigma$, and $l$.
		
	\textit{Step 5.}  
By Theorem~\ref{Thm C} in lower dimension, there exists a fixed $r \in \bN$ such that for any $(S, B_S), A_S$, we can find a very ample divisor $H_S$ on $S$ satisfying
\[
H_S^{d-1} \leq r, \quad (K_S + B_S) \cdot H_S^{d-2} \leq r, \quad \text{and} \quad H_S - A_S \ \text{is pseudo-effective}.
\]
By Lemma~\ref{lem:bir22 4.6}, we may further assume that $H_S - B_S$ is pseudo-effective.  

Since $A$ is big, there exists an effective $\bQ$-divisor $E$ such that $A \sim_{\bQ} E$.  
As $S$ is the pullback of a general element of a very ample linear system, we have $E_S := E|_S$ effective and $A_S \sim_{\bQ} E_S$.  
Moreover, 
\[
H_S - E_S \ \sim_{\bQ} \ H_S - A_S
\]
is also pseudo-effective.  

By the same argument as in Step~1, $v':=\Ivol(K_S + B_S)$ is fixed.  
Therefore, $(S, B_S)$ is $\epsilon$-lc for some $\epsilon \in \bQ^{> 0}$ depending only on $(d-1, \Phi, u, v')$ by Lemma~\ref{lem: finiteness of discrepancy}~(2).
		
		Thus by Theorem \ref{thm:bir21 1.8}, there is a ﬁxed $\tau  \in \bQ ^{>0}$ depending only on $d-1,\epsilon, r$ such that 
		$$\lct(S,B_S,|E_S|_{\bQ})>\tau,$$
		hence $(S, B_S + \tau E_S )$ is klt. Then by inversion of adjunction \cite[Theorem 5.50]{kollarBirationalGeometryAlgebraic1998}, 
		$(X, B + S + \tau E)$ is plt near $S$. Therefore,  $(X, B + \tau E)$ is lc over the complement of a ﬁnite set of closed points of $Z$: otherwise, the non-lc locus of $(X, B + \tau E)$ maps onto a  closed subset of $Z$ positive dimension which intersects $T$, hence $S$ intersects the non-lc locus of $(X, B + \tau E)$, a contradiction.
		
		\textit{Step 6.}   
		 In this step, we assume that $K_X+ B +\tau E$ is not nef.  Otherwise, $K_X+ B + \tau A\qlin K_X+ B + \tau E$ is nef, and we are done by Step 3.

		 Let $R$ be a $(K_X + B + \tau E)$--negative extremal ray, since $K_X + B + \tau E$ is nef and big over $Z$, $R$ is not contained in the fibers of   $X\to Z$. 
         By Step 5, the non-lc locus of $(X, B +\tau E)$ maps to finitely many points of $ Z$,
		so $R$ is not contained in  the image
		$$\mathrm{Im}(\overline{\NE}(\Pi) \to \overline{\NE}(X)),$$
		where $\Pi$ is the non-lc locus of $(X, B +\tau E)$. 
		
		Then by  the length of extremal ray \cite{ambroQuasilogVarieties2003} \cite[Theorem 1.1]{fujinoFundamentalTheoremsLog2011}, $R$ is generated by a curve $C$ with 
		$$(K_X + B + \tau E) \cdot C \geq -2d.$$
		Since $L\in|l(K_Z+B_Z+M_Z)|$ is very ample,  $f ^* L \cdot C =L\cdot f_*C\geq 1$,  we see that
		$$(K_X + B + 2df ^* L + \tau E) \cdot C \geq 0.$$	
		It follows  that
		$$K_X + B + 2df ^* L + \tau E$$	
		is nef. Since  $f ^* L \sim l(K_X + B)$, we see that
		$$ K_X +  B +\frac{\tau}{1+2dl} E \qlin  \frac{1}{1+2dl}\bigg(K_X +B+2df^*L+\tau E\bigg)$$
		is nef. 
		Hence after replacing $\tau $ with $\frac{\tau}{1+2dl}$, we can assume that $K_X + B + \tau E$ is nef.
	\end{proof}

    \subsection{Boundedness of pseudo-effective thereshold}
In this subsection, we show that the pseudo-effective threshold of $K_X+B$ with respect to $A$ is bounded for all 
\[
(X,B), \ A \to Z \ \in \ \cG_{klt}(d,\Phi,\Gamma,\sigma).
\]
 
	\begin{prop}\label{prop: bdd peff threshold}
		$\text{Theorem \ref{Thm A}}_{d-1}$  implies $\text{Theorem \ref{Thm B}}_{d}$.
	\end{prop}
	\begin{proof}
\textit{Step 0.}  
In this step, we introduce the top self-intersection function $\varsigma(t)$ and the volume function $\vartheta(t)$, and then outline the main idea of the proof using these functions.

Let 
\[
\varsigma(t) \in \mathbb{Q}[t], \quad 
\varsigma(t) := \big( A + t(K_X+B) \big)^d
= \sum_{i=0}^{d} \binom{d}{i} \, A^{\, d-i} \cdot (K_X + B)^{\, i} \, t^i,
\]
be the top self-intersection function.
It is easy to see that fixing $\varsigma$ is equivalent to fixing $\sigma$.  
Let 
\[
\vartheta(t) := \vol\big( A + t(K_X+B) \big)
\]
be the volume function. Then $\vartheta(t)$ is a non-negative, non-decreasing real function of $t$, and $\vartheta(t) = \varsigma(t)$ for $t \gg 0$.

It is enough to show that there exists a positive rational number $\tau$, depending only on $(d,\Phi,\Gamma,\sigma)$, such that 
\[
A + t(K_X+B) \ \text{is big for all } t > \tau.
\]
In other words, it suffices to show that $\vartheta(t) > 0$ for all $t > \tau$.

We will prove the proposition by showing:
\begin{itemize}
    \item There exists a positive rational number $\tau$, such that $\varsigma(t) > 0$ and strictly increasing for all $t \geq \tau$.
    \item Since $\varsigma(t) = \vartheta(t)$ for $t \gg 0$, a comparison of their derivatives shows that $\vartheta(t)$ decreases no faster than $\varsigma(t)$ as $t$ decreases. Hence, $\vartheta(t) \geq \varsigma(t) > 0$ for all $t \geq \tau$.
\end{itemize}
		
\textit{Step 1.} We prove this proposition by induction on the dimension of $Z$.  
Since $A^{\,d-i} \cdot (K_X+B)^{\,i} = 0$ for $i > \dim Z$, the dimension of $Z$ is determined by $\varsigma(t)$.  
Thus, we may assume $\dim Z = m$ is fixed.  
If $\dim Z = 0$, then clearly $A + t(K_X+B)$ is big.  
Hence, we may assume $\dim Z > 0$.  
        By Step~1 of the proof of Proposition~\ref{thm:bdd of tra smm with bdd peff threshold}, we may fix both
\[
u := \vol(A|_F) 
\quad\text{and}\quad 
v := \Ivol(K_X + B),
\]
where $F$ is a general fiber of $X \to Z$.

		If $\dim Z=1$, 
        then
		$$ \varsigma(t)=\big(A+t(K_X+B)\big)^d=A^d+dA^{d-1}\cdot(K_X+B)t=A^d+duvt.$$
		Let  $\varsigma^{\prime}(t)$ be the derivative of $\varsigma(t)$ with respect to $t$, it follows that
		$\varsigma^{\prime}(t)=duv$. 
		Since
		$$K_X+B\qlin f^*(K_Z+B_Z+M_Z)\qlin vF,$$   we have
		$$\vartheta(t)=\vol\big(A+t(K_X+B)\big)=v^d\vol(\frac{1}{v}A+tF).$$
		For each $t$ such that  $A+t(K_X+B)$ is big, i.e., $\vartheta(t)>0$, 
		we  may choose  a sufficiently general fiber $F_t$  of $X\to Z$ such that $F_t\nsubseteq \textbf{B}_+(\frac{1}{v}A+tF_t)$.
		Then by Theorem \ref{thm:LM09 cor 4.27},  the function $s\mapsto\vol(\frac{1}{v}A+tF_t+sF_t)$ is differentiable at $s=0$.  
	Let $\vartheta^{\prime}(t)$ denote the derivative of $\vartheta(t)$ with respect to $t$. This derivative is well-defined for all $t$ such that $\vartheta(t) > 0$. By Theorem \ref{thm:LM09 cor 4.27}, we have
		$$\frac{1}{v^d}\vartheta\Prime(t)=\frac{1}{v^d}\frac{\mathrm{d}  }{\mathrm{d} s}\vartheta(t+s)\bigg|_{s=0}=\frac{\mathrm{d}  }{\mathrm{d} s}\big(\vol(\frac{1}{v}A+tF_t+sF_t)\big)\bigg|_{s=0}=d\vol_{X|F_t}(\frac{1}{v}A+tF_t).$$
		It follows that for all $t$ such  that $\vartheta(t) > 0$,
		$$\vartheta^{\prime}(t)=dv^d\vol_{X|F_t}(\frac{1}{v}A+tF_t)\leq dv^d\vol\big((\frac{1}{v}A+tF_t)|_{F_t}\big)
		= dv^d\frac{1}{v^{d-1}}u=duv=\varsigma^{\prime}(t).$$
		
		\begin{figure}[ht]
			\caption{ The graph of $\varsigma(t)$ and $\vartheta(t)$ when $\dim Z=1$ }
			\label{figure:slope compare 1}
			\begin{center}
				\begin{tikzpicture}
					\begin{axis}[
						xlabel={$t$},
						ylabel={},
						grid=major,
						axis lines=middle,
						width=8cm,
						height=6cm,
						domain=-5:20,
						samples=100,
						legend pos=north west,
						xtick=\empty, 
						ytick=\empty, 
						]
						
						\addplot[domain=-5:20, samples=100, color=red, label={$\varsigma(t)$}]{0.5*(x-2)-4};
						
						\addplot[domain=-5:2, samples=2, color=black]{0};
						\addplot[domain=2:18, samples=100, color=blue, label={$\vartheta(t)$}]{0.0156*x^2-0.0625*x+0.0625};
						\addplot[domain=18:20, samples=100, color=black]{0.5*(x-2)-4};
						\addplot[color=black, mark=*,mark size=1] coordinates {(10, 0)};
						\addplot[color=black, mark=*,mark size=1] coordinates {(11, 0)};
						\addplot[color=black, mark=*,mark size=1] coordinates {(2, 0)};
						\addplot[color=black, mark=*,mark size=1] coordinates {(11, 1.3)};
                        \addplot[color=black, mark=*,mark size=1] coordinates {(18, 0)};
						\draw[dashed] (axis cs: 11, 1.2) -- (axis cs: 11, 0);
						\node[below] at (axis cs: 2, 0) {$\delta_X$};
						\node[below] at (axis cs: 9.9, 0) {$\alpha$};
						\node[below] at (axis cs: 11, 0) {$\tau$};
						\node[right,red] at (axis cs: 3, -3.5) {$\varsigma(t)$};
						\node[right,blue] at (axis cs: 4, 1) {$\vartheta(t)$};
                        \node[below] at (10, 5) {$\varsigma(t) = \vartheta(t)$ for $t \gg 0$};
                        \node[below] at (axis cs: 18, 0) {$\delta'_X$};
                        \draw[dashed] (axis cs: 18, 4) -- (axis cs: 18, 0);
                        \node[below] at (11, 3.2) {$\vartheta(\tau)$};
					
					\end{axis}
				\end{tikzpicture}
			\end{center}
		\end{figure}

		Let $\alpha$ be the root of $\varsigma(t)$ and set $\tau := \max\{\lceil \alpha \rceil + 1, 1\}$, so that $\tau$ is a positive rational number with $\varsigma(t) > 0$ for all $t \ge \tau$. Let $\delta_X$ be the largest real number such that $\vartheta(\delta_X) = 0$, where $\delta_X$ may \emph{a priori} depend on $X$. We claim that $\vartheta(\tau) > 0$. Suppose, for a contradiction, that $\vartheta(\tau) = 0$. Then $\tau \le \delta_X$, hence $\varsigma(\delta_X) \ge \varsigma(\tau) > 0$. Since $\vartheta(t) = \varsigma(t)$ for all $t \gg 0$, there exists $\delta'_X \gg 0$ (possibly depending on $X$) such that $\vartheta(\delta'_X) = \varsigma(\delta'_X)$. By \cite[Corollary~2.2.45]{lazarsfeldPositivityAlgebraicGeometry2004a}, the function $\vartheta(t)$ is continuous on $[\delta_X, \delta'_X]$, and since both $\varsigma(t)$ and $\vartheta(t)$ are differentiable on $(\delta_X, \delta'_X)$, Lemma~\ref{thm:rudin} yields some $\gamma_X \in (\delta_X, \delta'_X)$ such that $$\big(\vartheta(\delta'_X) - \vartheta(\delta_X)\big) \, \varsigma'(\gamma_X) = \big(\varsigma(\delta'_X) - \varsigma(\delta_X)\big) \, \vartheta'(\gamma_X).$$ Since $\vartheta(\delta'_X) = \varsigma(\delta'_X)$, $\vartheta(\delta_X) = 0$, and $\varsigma(\delta_X) > 0$, it follows that $\vartheta'(\gamma_X) > \varsigma'(\gamma_X)$, contradicting the inequality $\vartheta'(t) \le \varsigma'(t)$ for all $t > \delta_X$ from the previous paragraph. Therefore $\vartheta(\tau) > 0$, and hence $\vartheta(t) \ge \vartheta(\tau) > 0$ for all $t \ge \tau$.

		\textit{Step 2.} From now on we assume that $\dim Z = m > 1$. Recall that in Step~4 of the proof of Proposition~\ref{thm:bdd of tra smm with bdd peff threshold}, we pick a general element $T$ in the very ample linear system $|l(K_Z+B_Z+M_Z)|$ and let $S$ be its pullback to $X$, so that 
$$S \qlin l(K_X+B).$$
Define 
$$K_S + B_S := (K_X+B+S)|_S \quad \text{and} \quad A_S := A|_S,$$ 
so that 
$$K_S + B_S \qlin \Big(\frac{1}{l}+1\Big) S|_S.$$ 
Moreover, 
$$(S, B_S), A_S \to T \in \cG_{klt}(d-1, \Phi, \Gamma, \psi)$$ 
for some fixed polynomial $\psi(t) \in \bQ[t]$ depending only on $d, \Phi, \Gamma, \sigma$, with $\dim T = m-1$. By Theorem~\ref{Thm A} in lower dimension, there exists a positive rational number $\beta$, depending only on $d, \Phi, \Gamma, \sigma$, such that $A_S + t(K_S + B_S)$ is nef and big for all $t > \beta$.
		
		\textit{Step .} Recall that
$
\varsigma(t) = (A + t(K_X + B))^d$.
If $t > \beta(l+1)$, then $A_S + \frac{t}{l+1}(K_S + B_S)$ is nef and big by Step 2. We have
\begin{equation}\nonumber
\begin{aligned}
\varsigma^{\prime}(t) 
&= d \big(A + t(K_X + B)\big)^{d-1} \cdot (K_X + B) \\
&= \frac{d}{l} \big(A + t(K_X + B)\big)^{d-1} \cdot S \\
&= \frac{d}{l} \left(A_S + \frac{t}{l+1} (K_S + B_S)\right)^{d-1} \\
&> 0.
\end{aligned}
\end{equation}
Hence $\varsigma(t)$ is an increasing function on $\big(\beta(l+1), +\infty\big)$.

If $\varsigma(t)$ has no roots (which occurs only when $\dim Z$ is even), set $\tau = \beta(l+1) + 1$. If $\varsigma(t)$ has roots, let $\alpha$ be the largest root of $\varsigma(t)$ and set
$\tau = \max\{\beta(l+1), \lceil \alpha \rceil\} + 1$.
Note that $\tau$ is a positive rational number. Moreover, on $[\tau, +\infty)$, $\varsigma(t)$ is a positive, increasing real function, and $\vartheta(t)$ is a non-negative, non-decreasing real function.
		
		\begin{figure}[ht]
			\caption{ The graph of $\varsigma(t)$ and $\vartheta(t)$ when $\dim Z>1$ }
			\label{figure:slope compare 2}
			\begin{center}
			\begin{tikzpicture}
				\begin{axis}[
					axis lines=middle,
					xlabel={$t$},
					ylabel={},
					xmin=-1, xmax=14,
					ymin=-10, ymax=25,
					xtick=\empty,
					ytick=\empty,
					axis line style={-stealth},
					clip=false,
					width=12cm, height=8cm
					]
					
					%
					%
					
					%
					%
					%

					\addplot[domain=0:13.5, smooth, thick, red, samples=100] {(0.62*x-3)^2 - 3};

					\addplot[domain=-1:1, smooth, thick, black, samples=100] {0};
					\addplot[domain=1:12.7, smooth, thick, blue, samples=100] {1.2^(x+6) -0.5*(x+6)-0.1};
					%
					%
					\filldraw (1,0) circle (1pt) node[below] {};
					\filldraw (6,0) circle (1pt) node[below] {};
					\filldraw (7.66,0) circle (1pt) node[below] {};
					\filldraw (8.2,0) circle (1pt) node[below] {};
					\filldraw (8.2,6.1) circle (1pt) node[below] {};
					\filldraw (13,0) circle (1pt) node[below] {};
					\draw[dashed] (8.2,0) -- (8.2,6.1);
					\draw[dashed] (13,0) -- (13,22.5);
					
					\node[red] at (3, -4) {$\varsigma(t)$};
					\node[blue] at (4, 4) {$\vartheta(t) $};
					
					\node[below] at (8, 10) {$\vartheta(\tau)$};
					
					\node[below] at (8, 17) {$\varsigma(t) = \vartheta(t)$ for $t \gg 0$};
					\node[below] at (13, -0.1) {$\delta'_X$};
					\node[below] at (5.5, 0.4) {$\beta(\ell+1)$};
					\node[below] at (8.2, -0.1) {$\tau$};
					\node[below] at (7.6, -0.1) {$\alpha$};
					\node[below] at (1, -0.1) {$\delta_X$};
					
				\end{axis}
			\end{tikzpicture}
			\end{center}
		\end{figure}
		
		\textit{Step 4.}	In this step, we conclude the proof.
		We see that
		$$\vartheta(t)=\vol(A+t(K_X+B))=\frac{1}{l^d}\vol(lA+tS),$$
		for any $S\qlin l(K_X+B)$.
		For each $t$ such that  $A+t(K_X+B)$ is big, i.e., $\vartheta(t)>0$, 
		we  may choose  $S_t$ as the pullback of  a sufficiently general element  $T_t\in|l(K_Z+B_Z+M_Z)|$ such that $S_t\nsubseteq \textbf{B}_+(lA+tS_t)$.
		Then by Theorem \ref{thm:LM09 cor 4.27}, the function $s\mapsto\vol(lA+tS_t+sS_t)$ is differentiable at $s=0$.
		Let  $\vartheta^{\prime}(t)$ be the derivative of $\vartheta(t)$ with respect to $t$.  This derivative is well-defined for all $t$ such that $\vartheta(t) > 0$. By Theorem \ref{thm:LM09 cor 4.27}, we have
		$$l^d\vartheta\Prime(t)=
        l^d\frac{\mathrm{d}  }{\mathrm{d} s}\vartheta(t+s)\bigg|_{s=0}=
        \frac{\mathrm{d}  }{\mathrm{d} s}(\vol\big(lA+tS_t+sS_t)\big)\bigg|_{s=0}=d\vol_{X|{S_t}}(lA+tS_t).$$
		It follows that for all $t\geq \tau$ such that $\vartheta(t)>0$,  we have
		\begin{equation}\nonumber
			\begin{aligned}
				\vartheta^{\prime}(t)=&\frac{d}{l^d}\vol_{X|{S_t}}(lA+tS_t)  \\
				\leq&  \frac{d}{l}\vol\big((A+\frac{t}{l}S_t)|_{S_t}\big)\\
				= &  \frac{d}{l}\vol\big(A_{S_t}+\frac{t}{l+1}(K_{S_t}+B_{S_t})\big)\\
				=& \frac{d}{l}\big(A_{S_t}+\frac{t}{l+1}(K_{S_t}+B_{S_t})\big)^{d-1}\\
				=&\varsigma^{\prime}(t),
			\end{aligned}
		\end{equation}
		where the second-to-last equality follows from the fact that $A_{S_t}+\frac{t}{l+1}(K_{S_t}+B_{S_t})$ is nef on $[\tau, +\infty)$.
		
By the same argument as in the last paragraph of Step 1, we conclude that
$\vartheta(t) \geq \vartheta(\tau) > 0$ for all $t > \tau$.
	\end{proof}
We  use the following elementary result in the proof of Proposition~\ref{prop: bdd peff threshold}. 
Note that differentiability at the endpoints is not required.  

\begin{lemma}[{\cite[Theorem~5.9]{rudinPrinciplesMathematicalAnalysis1976}}]\label{thm:rudin}
Let $f$ and $g$ be continuous real-valued functions on $[a,b]$ that are differentiable on $(a,b)$. 
Then there exists a point $x \in (a,b)$ such that  
\[
\big(f(b)-f(a)\big) g'(x) = \big(g(b)-g(a)\big) f'(x).
\]
\end{lemma}
\begin{rem}
	In the case $\dim X=2$, by the Zariski decomposition for normal surfaces 
	\cite[Corollary 7.5]{sakaiWeilDivisorsNormal1984},
	the volume of a big divisor is greater than or equal to its self-intersection. 
	Thus, when $\dim X=2$, Proposition \ref{prop: bdd peff threshold} follows immediately from this fact. 
	However, this property does not necessarily hold in higher dimensions.
	For example, let $Y$ be a smooth 3-fold with $K_Y$ ample, and let 
	$\pi: X = \mathrm{Bl}_P Y \to Y$ be the blow-up of $Y$ at a closed point $P$. 
	Then $K_X = \pi^* K_Y + 2E$, where $E$ is the exceptional divisor over $P \in Y$, and $K_X$ is big. 
	It follows that $\vol(K_X) = \vol(K_Y) = (K_Y)^3$, while $(K_X)^3 = (K_Y)^3 + 8 E^3 = (K_Y)^3 + 8 > \vol(K_X)$.
\end{rem}
	\subsection{Boundedness of klt good minimal models}

	In this subsection, we prove the boundedness of klt good minimal models.
	
	\begin{prop}\label{thm:bdd of tra strongly smm}
    
		$\text{Theorem \ref{Thm A}}_{d}$ implies $\text{Theorem \ref{Thm C}}_{d}$.
	\end{prop}
	\begin{proof}
For each 
$$(X, B), A\to Z \in \cG_{klt}(d,\Phi, \Gamma,  \sigma),$$
by Step~1 of the proof of Proposition~\ref{thm:bdd of tra smm with bdd peff threshold}, we may fix both
\[
u := \vol(A|_F) 
\quad\text{and}\quad 
v := \Ivol(K_X + B),
\]
where $F$ is a general fiber of $X \to Z$.
By Lemma \ref{lem: finiteness of discrepancy} (2), $(X, B)$ is $\epsilon$-lc and $lB$ is a Weil divisor for some $\epsilon > 0$ and $l \in \bN$ depending only on $(d, \Phi, u, v)$. 
Replacing $l$ by a bounded multiple, Theorem \ref{Thm A} implies that
\[
L := l\Big(K_X+B+\frac{\tau}{2}A\Big)
\]
is a nef and big $\bQ$-Cartier Weil divisor. Let 
\[
L' := l(K_X+B) + L,
\] 
which is also a nef and big Weil divisor. Then $L' - K_X=(l-1)(K_X+B)+B+L$ is pseudo-effective. By \cite[Theorem 1.1]{birkarGeometryPolarisedVarieties2023}, there exists $m \in \bN$, depending only on $d$ and $\epsilon$, such that the linear system $|mL'|$ defines a birational map. Picking a general member $N \in |mL'|$, we have that $N \geq 0$ is a nef and big Weil divisor.
It then follows that
\[
\vol(K_X+B+N) = (2ml+1)^d \vol\Big(K_X+B+\frac{l\tau}{2(2ml+1)}A\Big) = (2ml+1)^d \, \sigma\Big(\frac{l\tau}{2(2ml+1)}\Big),
\]
which is fixed. Consequently, by Theorem \ref{thm:bir23 1.5}, the set of $(X, B+N)$ forms a bounded family.

Therefore, there exist a fixed $r \in \bN$ and a very ample divisor $H$ on $X$ such that
\[
H^d \leq r \quad \text{and} \quad H^{d-1} \cdot (K_X + B + N) \leq r.
\]
By Lemma \ref{lem:bir22 4.6}, $H - N$ is pseudo-effective. Since
\[
N - \frac{\tau}{2} A = (ml+1)(K_X + B) + (ml-1)\Big(K_X + B + \frac{\tau}{2} A\Big)
\]
is also pseudo-effective, it follows that $\frac{2}{\tau}H - A$ is pseudo-effective. Replacing $H$ by a bounded multiple, we may assume that $H-A$ is pseudo-effective.
	\end{proof}


	\begin{proof}[\textbf{Proof of Theorem }\ref{thm:boundedness}]
		This directly follows from Theorem \ref{Thm C}.
	\end{proof}

    \appendix
	\section{Moduli space}\label{sec:moduli}

In this appendix, we apply the boundedness results obtained in this paper to construct the moduli space of klt good minimal models of arbitrary Kodaira dimension, polarized by line bundles that are relatively ample over the bases of their respective Iitaka fibrations.
	
    We refer readers to \cite{alperStacksModuli2023} for the notions of  stacks, algebraic stacks,  Deligne-Mumford stacks and algebraic spaces.

	Let $d\in \bN$, $\Phi=\{a_1,a_2,\ldots,a_m\}$, where $a_i\in\bQ^{\geq 0}$, $\Gamma\subset \bQ^{\geq 0} $ be a finite set, and $\sigma\in \bQ[t]$ be a polynomial. In this appendix, we will fix these data.
	\subsection{Moduli functor of traditional stable minimal models}
Let $\mathbbm{k}$ be an algebraically  closed field of characteristic zero.
We define the main object studied in this appendix, as introduced in Birkar's survey note \cite[\S 10]{birkarBoundednessModuliAlgebraic2024}.
\begin{defn}[Traditional stable minimal models]\label{def:tsmm}

		A\emph{ traditional stable minimal model} $(X, B), A $  over $\mathbbm{k}$ consists of a projective connected pair $ (X, B) $ and a Cartier  divisor $ A$ (not necessarily effective) such that 
		\begin{itemize}
			\item $(X,B)$ is klt,
			\item $K_X+ B$ is semi-ample defining a contraction $ f : X \to Z$, and
			\item $K_X+ B + tA$ is ample for some $t > 0$.
		\end{itemize}
        
	A \textit{$(d,\Phi,\Gamma,\sigma)$-traditional stable minimal model} is a  traditional stable minimal model $(X, B), A$ such that
			\begin{itemize}
            \item $\dim X = d$, 
				\item the coeﬃcients of $B$ are in $\Phi$,
				\item $\vol(A |_F ) \in \Gamma$, where $F$ is any general ﬁber of $f : X \to Z$, and
				\item $(K_X + B + tA)^d = \sigma(t)$.
			\end{itemize}	
\end{defn}
	 We  recall the notion of relative Mumford divisor from \cite[Deﬁnition 4.68]{kollarFamiliesVarietiesGeneral2023}.
	\begin{defn}[Relative Mumford divisor]\label{def:relative mumford}
		Let $f:X\to S$ be a flat finite type morphism with $S_2$ fibers of pure dimension $d$. 
		A subscheme $D\subset X$ is a \emph{relative Mumford divisor}   if there is an open set $U\subset X$ such that

		\begin{itemize}
			\item ${\rm codim}_{X_s}(X_s\setminus U_s) \geq 2$ for each $s\in S$, 
			\item $D\vert_U$ is a relative Cartier divisor,
			\item $D$ is the closure of $D\vert_U$, and
			\item $X_s$ is smooth at the generic points of $D_s$ for every $s\in S$.
		\end{itemize}
		
		By $D|_U$ being relative Cartier we mean that $D|_U$ is a Cartier divisor on $U$ and that 
		its support does not contain any irreducible component of any fiber $U_s$.  
		
		If $D\subset X$  is a relative Mumford divisor for $f:X\to S$ and $T\to S$ is a morphism, then the \textit{divisorial pullback}  $D_T$ on $X_T := X \times_ S T$ is the relative Mumford divisor deﬁned to be the closure of the pullback of $D | _U $ to $U_T$.  In particular, for each $s \in S$, we deﬁne $D_s = D | _{X _s}$ to be the closure of $D |_ {U _s}$  which is the divisorial pullback of $D$ to $X_s$.
		
	\end{defn}

	\begin{defn}[Locally stable family]\label{def:locally stable}
		A \emph{locally stable family of klt  pairs} $(X,B) \to S$ over a reduced Noetherian scheme $S$
		is a flat finite type morphism $X\to S$ with $S_2$ fibers  and a $\bQ$-divisor $B$ on $X$   satisfying
		\begin{itemize}
			\item  each prime component of $B$ is a relative Mumford divisor,  
			\item $K_{X/S} +B$ is $\bQ$-Cartier, and
			\item  $(X_s,B_s)$ is a klt pair for any point $s\in S$.
		\end{itemize}
		
	\end{defn}
		%
		%

	We  define families of traditional minimal models and the corresponding moduli functor. 
	
	\begin{defn}
		Let $S$ be a reduced scheme over $\mathbbm{k}$.
		\begin{enumerate}
			\item  When $S = \Spec \mathbb{K}$ for a ﬁeld $\mathbb{K}$, we deﬁne a traditional stable minimal model over $\mathbb{K}$ as in Definition \ref{def:tsmm}   by replacing $\mathbbm{k}$ with $\mathbb{K}$ and replacing connected with geometrically connected. Similarly we can deﬁne $(d,\Phi,\Gamma,\sigma)$-traditional stable minimal models over $\mathbb{K}$.
			\item 	For general $S$, a \emph{family of traditional stable minimal models} over $S$ consists of a projective morphism $X\to S $ of schemes, a $\bQ$-divisor $B$ and  a line bundle $A$ on $X$ such that
			\begin{itemize}
				\item $(X,B)\to S$ is  a locally stable family,
				\item $(X_s,B_s),A_s$ is a traditional stable minimal model over $k(s)$ for every $s\in S$.
			\end{itemize}
			Here $X_s$  is the ﬁber of $ X \to S$ over $s$ and $B_s$ is the divisorial pullback  of $B$ to $X_s$. Moreover, $K_{X_s}+B_s$ is semi-ample which defines a contration $X_s\to Z_s$, and $A_s$ is a  line bundle on $X_s$  which is  ample over $Z_s$. We will denote this family by $(X, B), A \to S$.
			
			\item  Let $d\in \bN$, $\Phi=\{a_1,a_2,\ldots,a_m\}$, where $a_i\in\bQ^{\geq 0}$, $\Gamma\subset \bQ^{> 0} $ be a finite set, $\sigma\in \bQ[t]$ be a polynomial.
			A \emph{family of $(d, \Phi,\Gamma, \sigma)$-marked traditional stable minimal models} over $S$ is a family of traditional stable minimal models $(X, B), A \to S$ such that
			\begin{itemize}
				\item $B = \sum\limits_{i=1}^{m} a_iD_i$, where $D_i\geq 0$ are relative Mumford divisors, and
				\item $(X_s,B_s),A_s$ is a $(d, \Phi,\Gamma, \sigma)$-traditional stable minimal model over $k(s)$ for every $s\in S$, where $B_s = \sum\limits_{i=1}^{m} a_iD_{i,s}$.
			\end{itemize}
			
			\item We define the moduli functor 	$\mathfrak{TS}_{klt}(d,\Phi,\Gamma,\sigma)$ of $(d, \Phi,\Gamma, \sigma)$-traditional stable minimal models
			from the category of reduced  $\mathbbm{k}$-schemes to the category of groupoids 
			by choosing:
			\begin{itemize}
				\item On objects: 	for a reduced $\mathbbm{k}$-scheme $S$, one take			
				\begin{equation}\nonumber
					\begin{aligned}
					&	\mathfrak{TS}_{klt}(d,\Phi,\Gamma,\sigma)(S)\\=
					&	\{\text{family of } (d, \Phi,\Gamma, \sigma)\text{-traditional stable minimal models over } S\}. 					\end{aligned}
				\end{equation}
				We deﬁne an isomorphism $ (f\Prime:(X\Prime,B\Prime),A\Prime\to S) \to (f:(X,B),A\to S)$ of any two objects in $	\mathfrak{TS}_{klt}(d,\Phi,\Gamma,\sigma)(S)$ to be an isomorphism $\alpha_X:(X\Prime,B\Prime)\to (X,B)$  over $S$ such that  $A\Prime\sim_S \alpha_X^* A$. 
				\item On morphisms: $(f_T:(X_T, B_T),A_T \to T ) \to (f:(X, B),A \to S)$ consists of morphisms of reduced $\mathbbm{k}$-schemes $\alpha:T\to S$  such that the natural map $g:X _T \to X \times_S T $ is an isomorphism, $B_T$ is the divisorial pullback of $B$ and  $A_T\sim_T g^*\alpha_X^*A$. Here $\alpha_X:X \times_S T\to X$ is the base change of $\alpha$. 
				
			\end{itemize}
			
		\end{enumerate}
		
	\end{defn}
	
	Now we can state our main result on moduli.
	\begin{thm}\label{thm:moduli}
		$	\mathfrak{TS}_{klt}(d,\Phi,\Gamma,\sigma)$  is a separated Deligne-Mumford stack of finite type, which admits a coarse moduli space $TS_{klt}(d,\Phi,\Gamma,\sigma)$ as  a separated algebraic space.
	\end{thm}
	

	\subsection{Moduli stack of traditional stable minimal models}


	\begin{lemma}\label{lem:Hilbert poly}
		Let $\mathbb{K}$ be a field of characteristic zero.
		Then there exist  natural number $\tau$ and $I$ depending only on $(d,\Phi,\Gamma,\sigma)$ such that $\tau\Phi\subset\bN$ and they satisfy the following. For any $(X, B), A\in\mathfrak{TS}_{klt}(d, \Phi , \Gamma, \sigma)(K)$ and nef Cartier divisor $M$ on $X$, we have
		\begin{itemize}
			\item $\tau(K_X+ B)$ is a base point free divisor, $A+\tau(K_X+ B)$ is an ample Cartier divisor,
			\item Let $L_M:=I(A+\tau(K_X+ B))+M$, then $L_M$ is strongly ample, i.e. $L_M$ is very ample and
			$H ^q (X,kL_M) = 0$ for any $k, q > 0$,
		\end{itemize}
		
	\end{lemma}
	\begin{proof}
		 By the same argument as \cite[Proof of Lemma 10.2]{birkarModuliAlgebraicVarieties2022}, it is enough to find $\tau$ and $I$ when $\mathbb{K}= \bC$.
Note that $A$ is a line bundle in our setting. Hence, by the proof of Theorem \ref{thm:boundedness}, there exists $\tau \in \bN$ such that $\tau(K_X+B)$ is base point free, and both $A+(\tau-1)(K_X+B)$ and $A+\tau(K_X+B)$ are ample Cartier divisors. Applying the effective base point free theorem \cite[Theorem 1.1]{kollarEffectiveBasePoint1993} and the very ampleness lemma \cite[Lemma 7.1]{fujinoEffectiveBasepointfreeTheorem2017} to $A+\tau(K_X+B)$, we obtain $I_0 \in \bN$ such that $L_0 := I_0(A+\tau(K_X+B))$ is very ample.
		
		After replacing $I_0$ with a bounded multiple, we may assume that $L_0-(K_X+B)$ is nef and big.
		Let $I=(d+2)I_0$  and $\cF:=L_M- I_0(A+\tau(K_X+ B))$,
		then $$H^i(X,\cF\otimes L_0^{\otimes(-i)})=0$$ for all $i>0$ by Kawamata-Viehweg vanishing theorem.
		Thus $\cF$ is 0-regular with respect to $L_0$ (\cite[Definition 1.8.4]{lazarsfeldPositivityAlgebraicGeometry2004a}), and 
		hence $\cF$ is base point free by \cite[Theorem 1.8.5]{lazarsfeldPositivityAlgebraicGeometry2004a}. Therefore, 
		$$L_M=L_0+\cF$$ is very ample by \cite[Exercise \uppercase\expandafter{\romannumeral2} 7.5(d)]{hartshorneAlgebraicGeometry1977}. 
		Again we have $L_M-(K_X+B)$ is nef and big, hence 
		$H ^q (X,kL_M) = 0$ for any $k, q > 0$.  
	\end{proof}
	
		\begin{nota}\label{nota:}
		
		From now on, 	we will fix the positive natural numbers   $I$   and $\tau$ obtained in  Lemma \ref{lem:Hilbert poly}. Let $S$ be a reduced scheme,  for any $(f:(X, B), A\to S)\in\mathfrak{TS}_{klt}(d, \Phi , \Gamma, \sigma)(S)$, 
		we define 
		$$L_{1,S}:=I(A+\tau (K_{X/S}+ B))+I(A+\tau (K_{X/S}+ B))=2IA+2I\tau (K_{X/S}+ B),$$ 
			\begin{equation}\nonumber
			\begin{aligned}
				L_{2,S}:=&I(A+\tau (K_{X/S}+ B))+(I-1)(A+\tau (K_{X/S}+ B))+\tau (K_{X/S}+ B) \\
				= & (2I-1)A+2I\tau (K_{X/S}+ B)
			\end{aligned}
			\end{equation}
	and $L_{3,S}:=L_{1,S}+L_{2,S}$
		to be the divisorial sheaves on $X$.
		Then $L_{1,S}-L_{2,S}=A$, and $L_{j,S}$  are  strongly ample line bundles over $S$ for  $ j=1,2,3$ by Lemma \ref{lem:Hilbert poly} and the proof of Lemma \ref{lem:locally constant of hilbert poly}.
	\end{nota}

	\begin{lemma}\label{lem:locally constant of hilbert poly}
	
		Let $(X,B=\sum\limits_{i=1}^m a_iD_i),A\to S$ be a family of $(d, \Phi,\Gamma, \sigma$)-marked traditional stable minimal models over reduced Noetherian scheme $S$. For $ j=1,2,3$, let $L_{j,S}$ be the divisorial sheaves on $X$ as Notation \ref{nota:}. Then for every $k\in\bZ_{>0}$, the  functions
		$S\to \bZ$ by sending 
		\begin{enumerate} 
			\item $ s\mapsto h^0(X_s,kL_{j,s})$ for $ j=1,2,3$ and
			\item $ s\mapsto \mathrm{deg}_{L_{3,s}}(D_{i,s})$ for $i=1,2,\ldots,m$
		\end{enumerate} 
		are locally constant on $S$, where $L_{j,s}=L_{j,S}|_{X_s}$ and $D_{i,s}=D_i|_{X_s}$ are the divisorial pullbacks  to $X_s$, and $\mathrm{deg}_{L_{3,s}}(D_{i,s}):=D_{i,s}\cdot L_{3,s}^{d-1}$. 
	\end{lemma}
	\begin{proof}
		(1).
		For $j=1,2,3$, 
		it is enough to show that $L_{j,S}$ are flat over $S$: since then $\chi(X_s,kL_{j,s})$ are locally constant, and $L_{j,S}$ are strongly ample over $S$ by Lemma \ref{lem:Hilbert poly}, hence $h^0(X_s,kL_{j,s})$ are locally constant. Since $X\to S$ is flat, it  suffices to show that  $\cO_X(L_{j,S})$ are line bundles  by \cite[Proposition \uppercase\expandafter{\romannumeral3} 9.2(c)(e)]{hartshorneAlgebraicGeometry1977}. 
		
		Since $(X,B)\to S$ is a locally stable family, $B$ is a relative Mumford divisor over $S$, we see that $\tau(K_{X/S}+B)$ is $\bQ$-Cartier, and it is  mostly ﬂat (\cite[Deﬁnition 3.26]{kollarFamiliesVarietiesGeneral2023}) over $S$. Moreover, 
		  since $\cO_{X_s}(\tau(K_{X_s}+ B_s ))$ is a base point free line bundle  for any $s\in S$ by Lemma \ref{lem:Hilbert poly}, 
		 $\cO_X(\tau(K_{X/S}+ B) )$ is a mostly ﬂat family  of line bundles. Therefore, by \cite[Corollary 4.34 and Proposition 5.29]{kollarFamiliesVarietiesGeneral2023}, $\cO_X(\tau(K_{X/S}+ B) )$ is a  line bundle on $X$. 
		Furthermore, since $A$ is a line bundle on $X$,  $\cO_X(L_{j,S})$ are line bundles for $j=1,2,3$. 
		
		(2). It follows from \cite[Theorem 4.3.5]{kollarFamiliesVarietiesGeneral2023}.
	\end{proof}

Let  $n,l\in\bZ_{>0}$, $\textbf{c}=(c_1,c_2,\ldots,c_m)\in\bN^m$, and  $h\in \bQ[k]$ be a polynomial.   Let $S$ be a reduced scheme, for any $(f:(X, B=\sum\limits_{i=1}^m a_iD_i), A\to S)\in\mathfrak{TS}_{klt}(d, \Phi , \Gamma, \sigma)(S)$ and $j=1,2,3$, let $L_{j,S}$ be the strongly ample line bundles over $S$ as Notation \ref{nota:}. 
We define   $\mathfrak{TS}_{h,n,l,\textbf{c}}$  to be a full subcategory of $\mathfrak{TS}_{klt}(d,\Phi,\Gamma,\sigma)$ such that   $\mathfrak{TS}_{h,n,l,\textbf{c}}(S)$ is a groupoid whose objects consist of families of $(d, \Phi,\Gamma, \sigma)$-traditional stable minimal models over $S$ satisfying:
\begin{itemize}
	\item the  Hilbert polynomial of $X_s$ with respect to $L_{3,s}$ is $h$,
	\item $h^0(X_s,L_{1,s})-1=n$,
	\item $h^0(X_s,L_{2,s})-1=l$,  and
	\item $(\mathrm{deg}_{L_{3,s}}(D_{1,s}),\mathrm{deg}_{L_{3,s}}(D_{2,s}),\ldots,\mathrm{deg}_{L_{3,s}}(D_{m,s}))=\textbf{c}$
\end{itemize} 
for every $s\in S$.
\begin{lemma}\label{lem:subfunctor}

	We can write 
	$$\mathfrak{TS}_{klt}(d,\Phi,\Gamma,\sigma)=\bigsqcup_{h,n,l,\textbf{c}}\mathfrak{TS}_{h,n,l,\textbf{c}}$$
	as disjoint union, and  each $\mathfrak{TS}_{h,n,l,\textbf{c}}$ is a union of connected components of $\mathfrak{TS}_{klt}(d,\Phi,\Gamma,\sigma)$.
	Moreover, there are only finitely many $n,l\in \bZ_{>0}$, $\textbf{c}=(c_1,c_2,\ldots,c_m)\in\bN^m$ and $h \in\bQ [k]$ such that $\mathfrak{TS}_{h,n,l,\textbf{c}}$ is not  empty.
\end{lemma}

\begin{proof}
	Given any $(f:(X, B=\sum\limits_{i=1}^m a_iD_i), A\to S)\in\mathfrak{TS}_{klt}(d, \Phi , \Gamma, \sigma)(S)$.
	By  Lemma \ref{lem:locally constant of hilbert poly}, the Hilbert functions $$h_s(k)=\chi(X_s,kL_{3,s})=h^0(X_s,kL_{3,s})$$ of $X_s$ with respect to $L_{3,s}$, 
	and the numbers
	$$n_s=h^0(X_s,L_{1,s})-1,\ \ \ l_s=h^0(X_s,L_{2,s})-1\ \ \text{ and }\ \  c_{i,s}=\mathrm{deg}_{L_{3,s}}(D_{i,s})$$
	 are locally constant on $s\in S$ for all $1\leq i\leq m$. The first assertion follows from this fact.
	
	The second assertion follows from the fact that
	$n_s$, $l_s$, $c_{i,s} $ and $h_s$
	belong to a finite set for all $1\leq i\leq m$ by Theorem \ref{thm:boundedness} (these finiteness results can be reduced to the case when $s=\Spec \bC$ by the same argument as \cite[Proof of Lemma 10.2]{birkarModuliAlgebraicVarieties2022}). 
\end{proof}

	\begin{lemma}\label{lem: is stack}
		$\mathfrak{TS}_{h,n,l,\textbf{c}}$  is a stack.
	\end{lemma}
	\begin{proof}

		Since our argument follows the same strategy as in \cite[Proposition 2.5.14 and Example 2.5.9]{alperStacksModuli2023}, we only sketch the proof here.
		
				Axiom (1) of \cite[Deﬁnition 2.5.1]{alperStacksModuli2023}
		follows from  descent
		\cite[Proposition 2.1.7, Proposition 2.1.19, Proposition 2.1.4(1) and Proposition 2.1.16(2)]{alperStacksModuli2023}. 
		
		To verify Axiom (2) of \cite[Deﬁnition 2.5.1]{alperStacksModuli2023}, 
		i.e., given any descent datum  $(f\Prime,\xi)$
		 with respect to a covering $S\Prime\to S$ (see \cite[Remark 2.10]{hashizumeBoundednessModuliSpaces2025} for notions of covering and descent datum), where $(f\Prime:(X\Prime,B\Prime),A\Prime\to S\Prime )\in \mathfrak{TS}_{h,n,l,\textbf{c}}(S\Prime)$, we need to show
		 that $f\Prime$ descends to a family $(f:(X,B),A \to S )\in \mathfrak{TS}_{h,n,l,\textbf{c}}(S)$.  
		  We use the strongly $f\Prime$-ample line bundles  $\cO_{X\Prime}(L_{1,S\Prime}\Prime)$ and $\cO_{X\Prime}(L_{2,S\Prime}\Prime)$ as Notation \ref{nota:} instead of $\omega_{\cC\Prime/S\Prime}^{\otimes 3}$ in \cite[Proposition 2.5.14]{alperStacksModuli2023},  then the same argument as in \textit{loc.cit.} implies that $(X\Prime,B\Prime)\to S\Prime$  descends to $(X,B)\to S$. 
		Moreover, by applying 
		\cite[Proposition 2.1.4(2) and Proposition 2.1.16(2)]{alperStacksModuli2023} to the  covering $X\Prime\to X$,  we see that $A\Prime$  descends to a line bundle $A$ on $X$. Since every geometric ﬁber of $f:(X,B),A\to S$ is identiﬁed with a geometric ﬁber of $f\Prime:(X\Prime,B\Prime),A\Prime\to S\Prime$,  
		$(f:(X,B),A \to S )\in \mathfrak{TS}_{h,n,l,\textbf{c}}(S)$.
	\end{proof}

For any scheme $S$ and positive integer $n,l$, 
		Let $\bP^n_S\times_S\bP^l_S\cong \bP^n\times\bP^l\times S$ be the natural isomorphism, and
		$$\bP^n\stackrel{p_1}{\longleftarrow }\bP^n\times\bP^l\times S\stackrel{p_2}{\longrightarrow}  \bP^l$$
		be the projections. 
		Then for any $a,b\in\bZ$, we denote $p_1^*\cO_{\bP^n}(a)\otimes p_2^*\cO_{\bP^l}(b)$
		by	$\cO_{\bP^ n\times\bP^l\times S}(a,b)$. 
	
	\begin{thm}\label{thm: Artin stack}
		$\mathfrak{TS}_{h,n,l,\textbf{c}}$  is an algebraic stack of finite type.
	\end{thm}
	\begin{proof}

		\textit{Step 1.} 
		In this step, we  consider a suitable Hilbert scheme parametrizing  the total spaces  of interest.
		
		 For any $(f:(X,B),A\to S )\in\mathfrak{TS}_{h,n,l,\textbf{c}}(S)$ and for $j=1,2,3$, let $L_{j,S}$ be the strongly ample line bundles over $S$ as Notation \ref{nota:}. We get an embedding
		$$X\hookrightarrow\bP(f_*\cO_X(L_{1,S}))\times_S \bP(f_*\cO_X(L_{2,S})).$$
		We proceed to parametrize such embedding.
		
		Let $H=\Hilb_h (\bP^ n \times\bP^l)$ be the Hilbert scheme parametrizing closed subschemes of $\bP^n\times\bP^l $ with Hilbert polynomial $h$. Let
		$X_H=\Univ_h (\bP^ n\times\bP^l )\stackrel{i}{\hookrightarrow} \bP^n\times\bP^l\times H$ be the universal family over $H$,  and 
		$$\bP^n\stackrel{p_1}{\longleftarrow }\bP^n\times\bP^l\times H\stackrel{p_2}{\longrightarrow}  \bP^l.$$
		be the natural projections.
		Note that the $\text{PGL}_{n+1}\times\text{PGL}_{l+1}$ action on $\bP^n\times\bP^l$ induces a $\text{PGL}_{n+1}\times\text{PGL}_{l+1}$ action on $H$.
		Let $M_H:=i^*\cO_{\bP^ n\times\bP^l\times H}(1,1)$ and 
		$N_H:=i^*\cO_{\bP^ n\times\bP^l\times H}(1,-1)$ be the universal line bundles on $X_H$. 
		
		\textit{Step 2.} 
		In this step, we parametrize the boundary divisors in the moduli problem.	
		
		By \cite[Theorem 12.2.1 and Theorem 12.2.4]{grothendieckElementsGeometrieAlgebrique1966}, the locus $s\in H$ such that $X_s$ is geometrically connected and reduced, equidimensional, and geometrically normal is an open subscheme $H_1$ of $H$.
		
		Since $f_1:X_{H_1}\to H_1$ is equidimensional,
		and over reduced bases relative Mumford divisors are the same as K-ﬂat divisors \cite[Deﬁnition 7.1 and comment 7.4.2]{kollarFamiliesVarietiesGeneral2023},
		there is a separated $H_1$-scheme $\mathrm{MDiv}_{c}(X_{H_1}/H_1)$ of ﬁnite type which parametrizes relative Mumford divisors  of degree $c$  with respect to $M_{H_1}$ by \cite[Theorem 7.3]{kollarFamiliesVarietiesGeneral2023}.	
		Fixing  $\textbf{c}=(c_1,c_2,\ldots,c_m)\in\bN^m$, let $$H_2:=\mathrm{MDiv}_{c_1}(X_{H_1}/H_1)\times_{H_1}\mathrm{MDiv}_{c_2}(X_{H_1}/H_1)\times_{H_1}\cdots \times_{H_1}\mathrm{MDiv}_{c_m}(X_{H_1}/H_1)$$
		be the $m$-fold fiber product, we denote the universal family by
		$$(X_{H_2}, B_{H_2}=\sum\limits_{i=1}^m a_i D_{i,H_2}),N_{H_2}\to H_2,$$
		where $D_{i,H_2}$ are the universal families of relative Mumford divisors on $X_{H_2}$ of degree $c_i$  with respect to $M_{H_2}$ for $1\leq i\leq m$. 
		
		\textit{Step 3.} 
		By \cite[Theorem 4.8]{kollarFamiliesVarietiesGeneral2023}, there is a locally closed partial decomposition $H_3\to H_2$ satisfying the following: for any reduced scheme $W$ and morphism $q:W\to H_2$, then the family obtained by base change $f_W:(X_W,B_W)\to W$ is locally stable iff $q$ factors as $q:W\to H_3\to H_2$.
		
		Since $f_3:(X_{H_3},B_{H_3})\to H_3$ is locally stable,
		By \cite[Theorem 4.28]{kollarFamiliesVarietiesGeneral2023}, there is a locally closed partial decomposition $H_4\to H_3$ satisfying the following: for any reduced scheme $W$ and morphism $q:W\to H_3$, the divisorial pullback of $\tau(K_{X_{H_3}/H_3}+ B_{H_3})$ to $W \times _{H_3}X_{H_3}$ is Cartier iff $q$ factors as $q:W\to H_4\to H_3$.
		
		\textit{Step 4.} 
		Since the fibers $X_s$ of $f_4:X_{H_4}\to H_4$ are  reduced and connected by Step 2, we have $h^0(X_s,\cO_{X_s})=1$.
		Since $\tau(K_{X_{H_4}/H_4}+ B_{H_4})$  is Cartier by Step 3,
		by \cite[Lemma 1.19]{viehwegQuasiprojectiveModuliPolarized1995}, there is a locally closed subscheme $H_5\subset H_4$ with the following property: for any  scheme $W$ and morphism $q:W\to H_4$, 
		$$\cO_{X_W}(1,0)\sim_W N_W^{2I}\otimes \omega_{X_W/W}^{[2I\tau]}(2I\tau B_W)\text{ and } $$
		$$\cO_{X_W}(0,1)\sim_W N_W^{2I-1}\otimes \omega_{X_W/W}^{[2I\tau]}(2I\tau B_W)$$
		iff $q$ factors as $q:W\to H_5\to H_4$,
		where $\cO_{X_W} (1,0)$ and  $\cO_{X_W} (0,1)$ are the pullbacks of 
		$\cO_{\bP^ n\times\bP^l\times H_4}(1,0)$ and $\cO_{\bP^ n\times\bP^l\times H_4}(0,1)$ 
		to $X_W$, respectively. 
		
		\textit{Step 5.} 
		In this step, we cut the locus parametrizing $(d,\Phi,\Gamma,\sigma)$-traditional stable minimal models.
		
		(1). By \cite[Lemma 8.5]{birkarModuliAlgebraicVarieties2022}, there is  a locally closed subscheme $H_6\subset H_5$ such that for any $s\in H_6$, $K_{X_s}+ B_s$ is semi-ample defining a contraction $X_ s\to Z_s$.
		
		(2). Since ampleness and klt are open conditions, there is an open subscheme $H_7\subset H_6$ such that  $N_s+ \tau(K_{X_s}+B_s)$ is ample and $(X_s,B_s)$ is klt for  any $s\in H_7$.
		
		(3). By \cite[Lemma 8.7]{birkarModuliAlgebraicVarieties2022} (the condition of $N_s$ being effective is not required in the proof), there is  a locally closed subscheme $H_8\subset H_7$ such that for any $s\in H_8$, $\vol(N_s|_F)\in\Gamma$ for the general ﬁbres $F$ of $X_ s\to Z_s$.
		
	(4).	For each $s\in H_8$, since $K_{X_s}+ B_s$ is semi-ample and $N_s+ \tau(K_{X_s}+B_s)$ is ample,
		$K_{X_s}+B_s+tN_s$ is ample for each $t\in (0, \frac{1}{\tau}]$, then
		$$\theta_s(t)=\vol(K_{X_s}+ B_s + t N_s)=(K_{X_s}+ B_s + t N_s)^d$$
		is a polynomial in $t$ of degree $\leq d$ on the interval $(0, \frac{1}{\tau}]$. By Step 3(\romannumeral4)  of \cite[ Proof of Proposition 9.5]{birkarModuliAlgebraicVarieties2022}, 
		there is an open and closed subscheme $H_{9}\subset H_8$ such that $\theta_s(t)=\sigma(t)$ on the interval $(0, \frac{1}{\tau}]$. 
		
		Therefore,   $f_9:(X_{H_9}\subset \bP^n\times\bP^l\times H_9, B_{H_9}),N_{H_9}\to H_9$ is a family of $(d,\Phi,\Gamma,\sigma)$-traditional stable minimal models. For $j=1,2$, let $L_{j,H_9}$ be the strongly ample line bundles over $H_9$ as Notation \ref{nota:}. 
		Then $f_{9*}\cO_{X_{H_9}}(L_{1,H_9})$ and $f_{9*}\cO_{X_{H_9}}(L_{2,H_9})$ are locally free sheaves of rank $n+1$ and $l+1$, respectively. 
			Shrinking $H_9$, we may assume that they are free sheaves, and hence
		$$\bP(f_{9*}\cO_{X_{H_9}}(L_{1,H_9}))\cong\bP^n_{H_9} \text{ and   }\bP(f_{9*}\cO_{X_{H_9}}(L_{2,H_9}) )\cong\bP^l_{H_9}.$$

		\textit{Step 6.} In this step, we  will prove that
		$$\mathfrak{TS}_{h,n,l,\textbf{c}}\cong[H_{9}/\text{PGL}_{n+1}\times\text{PGL}_{l+1}].$$
		Then since $H_{9}$ is a ﬁnite type scheme and $[H_{9}/\text{PGL}_{n+1}\times\text{PGL}_{l+1}]$ is an algebraic stack,  $\mathfrak{TS}_{h,n,l,\textbf{c}}$
		is a finite type algebraic stack. 
		
		
		We follow the arguments of \cite[Theorem 3.1.17]{alperStacksModuli2023} and \cite[Proposition 3.9]{ascherModuliBoundaryPolarized2023}.
	By our construction,   the  universal family
	$f_9:(X_{H_{9}}\subset \bP^n\times\bP^l\times H_{9}, B_{H_{9}}),N_{H_{9}}\to H_{9}$
	is an object in $\mathfrak{TS}_{h,n,l,\textbf{c}}(H_{9})$, 
which induces  a morphism $H_9\to \mathfrak{TS}_{h,n,l,\textbf{c}}$, where this morphism just forgets the projective embeddings.
		Moreover, this morphism is $\text{PGL}_{n+1}\times\text{PGL}_{l+1}$-invariant, hence
		descends to a morphism $\Psi^{\text{pre}}:[H_{9}/\text{PGL}_{n+1}\times\text{PGL}_{l+1}]^{\text{pre}}\to \mathfrak{TS}_{h,n,l,\textbf{c}}$ of prestacks. 
		Since $\mathfrak{TS}_{h,n,l,\textbf{c}}$ is a stack by Lemma \ref{lem: is stack}, the universal property of stackiﬁcation \cite[Theorem 2.5.18]{alperStacksModuli2023} yields a morphism $\Psi:[H_9/\text{PGL}_{n+1}\times\text{PGL}_{l+1}]\to \mathfrak{TS}_{h,n,l,\textbf{c}}$. 
		

		To construct the inverse, 
		consider $(f:(X,B),A\to S)\in  \mathfrak{TS}_{h,n,l,\textbf{c}}(S)$, 
		since $f_*\cO_X(L_{1,S})$ and $f_*\cO_X(L_{2,S})$ are  locally free by Step 1, there exists an open cover $S=\cup_i S_i$ over which their restrictions are free. 
		Choosing trivializations induce embeddings $g_i:(X_{S_i},B_{S_i})\hookrightarrow \bP^n\times\bP^l\times S_i$.
Moreover,  	  we have  $A_{S_i}\sim_{S_i} N_{S_i}:=  g_i^*\cO_{\bP^n\times\bP^l\times S_i}(1,-1)$.		   
		    Hence by our construction of $H_{9}$, we have morphisms $\Phi_i:S_i\to H_{9}$. 
		Over the intersections $S_i\cap S_j$, the trivializations diﬀer by a section $s_{ij}\in H^0(S_i\cap S_j,\text{PGL}_{n+1}\times\text{PGL}_{l+1})$.
		Therefore the $\Phi_i$ glue to a morphism $\Phi :S\to[H_{9}/\text{PGL}_{n+1}\times\text{PGL}_{l+1}]$, 
		which induces a morphism $\mathfrak{TS}_{h,n,l,\textbf{c}}\to[H_{9}/\text{PGL}_{n+1}\times\text{PGL}_{l+1}]$, that is the inverse of $\Psi$.
	\end{proof}

	\subsection{Moduli space of traditional stable minimal models}
%
%
%
%
		We need the following  separatedness result to obtain the  coarse moduli space of traditional stable minimal models. 
	
%

	
	\begin{thm}\label{thm:separated}
	 Let  $f : (X, B), A \to C$ and $f\Prime  : (X\Prime , B\Prime), A\Prime \to C$ be two families of $(d,\Phi,\Gamma,\sigma)$-traditional stable minimal models over a smooth  curve $ C$.
Let $0 \in C$ be a closed point and $C^o := C \setminus \{ 0 \}$ the punctured curve. 
		Assume there exists an isomorphism 
		$$g^ o : ((X , B) , A) \times_C C^ o \to ((X\Prime , B\Prime ), A \Prime ) \times _C C ^o $$
		over $C^o$,
		then $g ^o$ can be extended to an isomorphism $g : (X , B  ) , A\to (X \Prime , B\Prime ) , A \Prime$ over $C$.
	\end{thm}
	\begin{proof}
		Consider  $L:=A\Prime+\tau(K_{X/C}+B)$ and $L\Prime:=A\Prime+\tau(K_{X\Prime/C}+B\Prime)$, where $\tau$ is the positive natural number as Lemma \ref{lem:Hilbert poly}.  By the proof of   Lemma \ref{lem:locally constant of hilbert poly}, $L$  is an $f$-ample  Cartier divisor on $X$ (resp.  $L\Prime$ is an $f\Prime$-ample  Cartier divisor on $X\Prime$).
		Let $g : X \dasharrow X \Prime$ be the birational map induced by $g^o$,
		then by the  same argument as in  \cite[Proof of Proposition 4.4]{hashizumeBoundednessModuliSpaces2025},
		$g$ is an isomorphism over $C$.
	\end{proof}
	\begin{cor}\label{cor:aut is finite}
		For any $(X, B), A\in\mathfrak{TS}_{klt}(d, \Phi , \Gamma, \sigma)(\mathbbm{k})$,  $\Aut((X, B), A)$ is ﬁnite.
	\end{cor}
	\begin{proof}
		It  follows from Theorem \ref{thm:separated} and the argument of  \cite[Proof of Corollary 3.5]{blumUniquenessKpolystableDegenerations2019}. 
	\end{proof}
	
		\begin{proof}[\textbf{Proof of Theorem }\ref{thm:moduli}]
			By	Theorem \ref{thm: Artin stack} and Lemma \ref{lem:subfunctor},  $\mathfrak{TS}_{klt}(d,\Phi,\Gamma,\sigma)$ is an algebraic  stack of finite type.
			By Corollary \ref{cor:aut is finite} and \cite[Theorem 3.6.4]{alperStacksModuli2023}, $\mathfrak{TS}_{klt}(d,\Phi,\Gamma,\sigma)$  is a Deligne-Mumford stack.
			Moreover, Theorem \ref{thm:separated}  and \cite[Theorem 3.8.2(3)]{alperStacksModuli2023} imply that $\mathfrak{TS}_{klt}(d,\Phi,\Gamma,\sigma)$ is a separated Deligne-Mumford stack of finite type. Therefore, we may apply the Keel–Mori's theorem \cite{keelQuotientsGroupoids1997}\cite[Theorem 4.3.12]{alperStacksModuli2023} to see that $\mathfrak{TS}_{klt}(d,\Phi,\Gamma,\sigma)$ has a coarse moduli space $TS_{klt}(d,\Phi,\Gamma,\sigma)$, which is a separated algebraic space.
		\end{proof}
		%
		

	\bibliographystyle{alphaurl}
		\bibliography{klt-smm}
		
	\end{document}